\documentclass[11pt,a4paper]{article}

\pdfoutput=1

\usepackage{amssymb,amsmath,amsthm}
\usepackage{mathrsfs}
\usepackage{braket,color}
\usepackage{xypic}

\theoremstyle{plain}
\newtheorem{thm}{Theorem}[section]
\newtheorem{prop}[thm]{Proposition}
\newtheorem{lem}[thm]{Lemma}
\newtheorem{cor}[thm]{Corollary}

\theoremstyle{definition}
\newtheorem{dfn}[thm]{Definition}
\newtheorem{rem}[thm]{Remark}
\newtheorem{exam}[thm]{Example}

\numberwithin{equation}{section}

\makeatletter
\renewenvironment{proof}[1][\proofname]{\par
  \pushQED{\qed}%
  \normalfont \topsep6\p@\@plus6\p@\relax
  \trivlist
  \item[\hskip\labelsep
	\bfseries
    #1\@addpunct{.}]\ignorespaces
}{%
  \popQED\endtrivlist\@endpefalse
}
\makeatother

\DeclareMathOperator{\ev}{ev}
\DeclareMathOperator{\id}{id}
\DeclareMathOperator{\tr}{tr}
\DeclareMathOperator{\str}{str}

\DeclareMathOperator{\mult}{mult}

\newcommand{\frg}{\mathfrak{g}}

\newcommand{\frz}{\mathfrak{z}}

\newcommand{\rect}{\mathcal{W}^{\kappa}(\mathfrak{g},f)}

\newcommand{\ve}{\varepsilon}

\newcommand{\bbZ}{\mathbb{Z}}
\newcommand{\bbC}{\mathbb{C}}

\newcommand{\frakS}{\mathfrak{S}}

\newcommand{\affY}{Y(\affsl)}

\newcommand{\compaffY}{\affY_{\mathrm{comp}}}

\newcommand{\gl}{\mathfrak{gl}}
\newcommand{\fraksl}{\mathfrak{sl}}
\newcommand{\affsl}{\hat{\mathfrak{sl}}_n}

\newcommand{\pari}[1]{\overline{#1}}
\newcommand{\cpari}[2]{\chi({#1},{#2})}

\title{Coproduct for affine Yangians and parabolic induction for rectangular $W$-algebras}

\author{Ryosuke Kodera and Mamoru Ueda}

\date{}

\makeatletter
\let\@old@@maketitle=\@maketitle
\def\@maketitle{%
\footnotetext{%
\hspace*{-1em}\hspace*{-\footnotesep}%
R.K. Department of Mathematics and Informatics,
Graduate School of Science, 
Chiba University\\
E-mail address: kodera@math.s.chiba-u.ac.jp\\
M.U. Research Institute for Mathematical Sciences,
Kyoto University\\
E-mail address: udmamoru@kurims.u-kyoto.ac.jp
}
\@old@@maketitle
}
\makeatother

\begin{document}
\maketitle

\begin{abstract}
We construct algebra homomorphisms from affine Yangians to the current algebras of rectangular $W$-algebras both in type A.
The construction is given via the coproduct and the evaluation map for the affine Yangians.
As a consequence, we show that parabolic inductions for representations of the rectangular $W$-algebras can be regarded as tensor product representations of the affine Yangians under the homomorphisms.
The same method is applicable also to the super setting.
\end{abstract}

\section{Introduction}

$W$-algebras form an important class of vertex algebras.
They have been widely studied with connections to various areas of mathematics and mathematical physics such as two-dimensional conformal field theories, integrable systems, and four-dimensional gauge theories.
We refer to Arakawa's ICM talk \cite{MR3966808} and references therein for recent developments.
The goal of this paper is to relate $W$-algebras for rectangular nilpotent elements in type A to certain quantum groups, called affine Yangians.  

The $W$-algebra $\rect$ is defined by the quantized Drinfeld-Sokolov reduction \cite{MR1071340, MR2013802} 
from the following data:
a reductive Lie algebra $\frg$ over $\bbC$,  
an invariant symmetric bilinear form $\kappa$ on $\frg$,
and a nilpotent element $f$ in $\frg$.
Here we omit to mention a choice of a grading of $\frg$.
It admits the Miura map, an injective homomorphism from $\mathcal{W}^{\kappa}(\frg,f)$ to the universal affine vertex algebra $V^{\tilde\kappa}(\frg_{0})$ where $\frg_{0}$ denotes a Levi subalgebra of $\frg$ determined by a grading of $\frg$, and $\tilde{\kappa}$ is a bilinear form on $\frg_{0}$ induced from $\kappa$ with an appropriate shift.
Thus we can regard $\rect$ as a vertex subalgebra of  $V^{\tilde\kappa}(\frg_{0})$.

In this paper, we consider the case $\frg = \gl_{N}$ for $N=ln$ with $l \geq 1$ and $n \geq 3$, and $f$ is a nilpotent element whose Jordan form corresponds to the partition $(l^{n})$.
We call them rectangular $W$-algebras because of the shape of the Young diagram corresponding to the partition $(l^{n})$.
Recently the rectangular $W$-algebras have been intensively studied; see \cite{MR3925243, MR4032024, MR4061286, MR4088277} for instance. 

For a suitable choice of $\kappa$ (see Section~\ref{section:Rectangular_W-algebras}), the universal affine vertex algebra $V^{\tilde\kappa}(\frg_{0})$ is isomorphic to the tensor product of $l$ copies of the universal affine vertex algebra of $\gl_{n}$.
Arakawa-Molev \cite{MR3598875} constructed elements $W_{i,j}^{(r)}$ ($r=1,\ldots,l$ and $i,j=1,\dots,n$) in $V^{\tilde\kappa}(\frg_{0}) \cong V^{\tilde\kappa}(\gl_{n})^{\otimes l}$ and proved that they generate the rectangular $W$-algebra $\rect$.

Similar, and more refined story has been known for the corresponding finite $W$-algebras.
The rectangular finite $W$-algebra for $\frg = \gl_{N}$ and $f = (l^{n})$ is regarded as a subalgebra of $U(\frg_{0}) \cong U(\gl_{n})^{\otimes l}$.
Ragoucy-Sorba \cite{MR1700166} first noticed a relation between the rectangular finite $W$-algebras and Yangians of finite type both in type A.
Later Brundan-Kleshchev \cite{MR2199632} extended Ragoucy-Sorba's result; they constructed surjective algebra homomorphisms from shifted Yangians to finite $W$-algebras for arbitrary nilpotent elements in $\gl_{N}$ and determined their kernels.
In \cite[Section~12]{MR2199632}, it was observed that in the case of rectangular type, the homomorphism is related to the coproduct and the evaluation map for the Yangian $Y(\gl_{n})$.

Our main theorem is an affine analog of this observation, which asserts that the composition of the coproduct and the evaluation map for the affine Yangian $\affY$ gives an algebra homomorphism to the current algebra of the rectangular $W$-algebra $\rect$.
It answers affirmatively to an question raised by Genra in \cite[Introduction]{MR4091897}.

Let us state the main result and its consequence.
Let $k$ be the level of the rectangular $W$-algebra $\rect = \mathcal{W}^{\kappa}(\mathfrak{gl}_{N}, (l^n))$ where $N=ln$ with $l \geq 1$ and $n \geq 3$.
We associate with a vertex algebra $V$, its current algebra $\mathfrak{U}(V)$.
Let $\hbar, \ve$ be the parameters of the affine Yangian $\affY$.
We assume that the condition $k+N=-\ve/\hbar$ is satisfied.
Then we will define a map $\Phi_{l}$ in Definition~\ref{dfn:map}, which is an algebra homomorphism from a completed affine Yangian $\affY_{\mathrm{comp}}$ to a completed tensor product $U(\hat{\gl}_{n}^{\alpha})_{{\mathrm{comp}}}^{\otimes l}$ of the universal enveloping algebras of the affine Lie algebras $\hat{\gl}_{n}^{\alpha}$ at level $\alpha = k+N-n$.
Let $\Delta$ denote the coproduct of $\affY$.
Following \cite[Section~7.2]{MR4091897}, we will introduce in Section~\ref{section:Coproduct_and_parabolic_induction} an injective algebra homomorphism
\[
	\tilde{\Delta}_{l_1,l_2} \colon \mathfrak{U}(\rect) \to \mathfrak{U}(\mathcal{W}^{\kappa_1}(\mathfrak{gl}_{N_1}, f_1)) \otimes_{\mathrm{comp}} \mathfrak{U}(\mathcal{W}^{\kappa_2}(\mathfrak{gl}_{N_2}, f_2))
\]
with $l=l_1+l_2$ and $N_i=l n_i$, $f_i=(l_i^n)$.

\begin{thm}[Theorem~\ref{thm:main}, Corollary~\ref{cor:parabolic_induction}]\label{thm:introduction_main}
Assume $k+N=-\ve/\hbar$.
Then the map $\Phi_{l}$ gives an algebra homomorphism
\[
	\Phi_{l} \colon \affY_{\mathrm{comp}} \to \mathfrak{U}(\rect)
\]
which makes the following diagram commutative:
\[
	\xymatrix{
		\affY_{\mathrm{comp}} \ar[d]_{\Delta} \ar[r]^{\Phi_{l}} \ & \ \mathfrak{U}(\rect) \ar[d]^{\tilde{\Delta}_{l_1,l_2}} \\
		\affY^{\otimes 2}_{\mathrm{comp}} \ar[r]^-{\Phi_{l_1} \otimes \Phi_{l_2}} \ \ & \ \ \mathfrak{U}(\mathcal{W}^{\kappa_1}(\mathfrak{gl}_{N_1}, f_1)) \otimes_{\mathrm{comp}} \mathfrak{U}(\mathcal{W}^{\kappa_2}(\mathfrak{gl}_{N_2}, f_2)).
	}	
\]
Moreover, $\Phi_{l}$ is surjective if $k+N-n \neq 0$.
\end{thm}

Given representations $M_i$ of $\mathcal{W}^{\kappa_i}(\mathfrak{gl}_{N_i}, f_i)$ for $i=1,2$, we regard $M_1 \otimes M_2$ as a representation of $\mathfrak{U}(\rect)$ via $\tilde{\Delta}_{l_1,l_2}$.
This is called parabolic induction.
By Theorem~\ref{thm:introduction_main}, each $M_1$, $M_2$, and $M_1 \otimes M_2$ are regarded as representations of the affine Yangian $\affY$ via $\Phi_{l_1}$, $\Phi_{l_2}$, and $\Phi_{l}$ respectively.
The commutative diagram in Theorem~\ref{thm:introduction_main} means that $M_1 \otimes M_2$ is the tensor product representation via the coproduct $\Delta$.

We obtain a similar result for the super setting in Theorem~\ref{thm:introduction_main_super}.
The second named author constructed in \cite{Affine_super_Yangians_and_rectangular_W-superalgebras} an algebra homomorphism $\Phi$ from $\affY_{\mathrm{comp}}$ to $\mathfrak{U}(\rect)$ directly.
Namely, he assigned elements of $\mathfrak{U}(\rect)$ to generators of $\affY$ and checked that it respects the defining relations of $\affY$.
His construction is applicable also to the super setting.
The map $\Phi_{l}$ coincides with $\Phi$ up to a twist in both the non-super and the super case (see Remark~\ref{rem:difference} (i) for the twist in the non-super case).
Hence Theorem~\ref{thm:introduction_main} and Theorem~\ref{thm:introduction_main_super} give an alternative proof of the main result of \cite{Affine_super_Yangians_and_rectangular_W-superalgebras}.

Our work is motivated by the Alday-Gaiotto-Tachikawa (AGT) correspondence for parabolic sheaves on $\mathbb{P}^{1} \times \mathbb{P}^{1}$.
There are various versions of AGT correspondence and we only refer to the original paper \cite{MR2586871} and \cite[Introduction, 1.5 and 1.11 (4)]{MR3592485} here.
It has been shown that the affine Yangian $\affY$ acts on the localized equivariant cohomology of affine Laumon spaces (framed moduli spaces of parabolic sheaves) by Feigin-Finkelberg-Negut-Rybnikov \cite{MR2827177} and Finkelberg-Tsymbaliuk \cite{MR3971731}.
We conjecture that our map $\Phi_{l}$ induces an action of $\rect$ on the equivariant cohomology of affine Laumon spaces and that the action is one predicted by the AGT correspondence.
This conjecture would give an affine analog of results of Braverman-Feigin-Finkelberg-Rybnikov \cite{MR2851149} and Nakajima \cite{MR3024827} in the case of rectangular type.
We note that Negut studies the $K$-theory version in \cite{Toward_AGT_for_parabolic_sheaves, Deformed_W_algebras_in_type_A_for_rectangular_nilpotent} by a different approach. 

We finish this section with comments on the case for $n=1,2$.
Works of Maulik-Okounkov \cite{MR3951025} and Schiffmann-Vasserot \cite{MR3150250} related the affine Yangian $Y(\hat{\mathfrak{gl}}_1)$ to the principal $W$-algebras of type A, and established the AGT correspondence in this case.
We are missing the $n=2$ case at the moment since neither of coproduct nor evaluation map for $Y(\hat{\mathfrak{sl}}_2)$ has not been defined.

This paper is organized as follows.
In Section~\ref{section:Rectangular_W-algebras}--\ref{section:Parabolic_induction}, we prepare basics on the rectangular $W$-algebras.
In Section~\ref{section:Affine_Yangians}--\ref{section:Evaluation_map}, we recall the coproduct and the evaluation map for the affine Yangians.
We construct the algebra homomorphisms from the affine Yangians to the rectangular $W$-algebras in Section~\ref{section:Affine_Yangians_and_rectangular_W-algebras} and relate the parabolic induction for the rectangular $W$-algebras to the coproduct for the affine Yangians in Section~\ref{section:Coproduct_and_parabolic_induction}. 
We discuss some commuting elements of the rectangular $W$-algebras in Section~\ref{section:Commuting_elements}.
In Section~\ref{section:Super_case}, we give a brief explanation for the super case. 
Appendix A, B, and C collect some computations.

\subsection*{Acknowledgments}
The authors are grateful to Tomoyuki Arakawa, Boris Feigin, Ryo Fujita, Naoki Genra, 
Toshiro Kuwabara, Andrew Linshaw, Hiraku Nakajima, Shigenori Nakatsuka, Andrei Negut, Masatoshi Noumi, Shoma Sugimoto, Husileng Xiao, Yasuhiko Yamada, and Shintarou Yanagida for valuable discussions and suggestions.

Some part of results of this paper were presented by the first named author in Workshop on 3d Mirror Symmetry and AGT Conjecture held in October 21--25, 2019 at Institute for Advanced Study in Mathematics, Zhejiang University, Hangzhou.
He thanks to their hospitality.

The first named author was supported by JSPS KAKENHI Grant Number 18K13390, 21K03155.
His work was also supported in part by JSPS Bilateral Joint Projects (JSPS-RFBR collaboration) ``Elliptic algebras, vertex operators and link invariants'' from MEXT, Japan.
The second named author was supported by Grant-in-Aid for JSPS Fellows 20J12072.

\section*{Notations for parameters}

Throughout the paper, we fix nonzero positive integers $l$ and $n$, then put $N=ln$.
We assume $n \geq 3$ from Section~\ref{section:Affine_Yangians} to Section~\ref{section:Commuting_elements}, in particular in the main theorem.

We fix a complex number $k$, which stands for the level of the $W$-algebra.
We also use an alternative parameter $\alpha$ defined by $\alpha= k+(l-1)n= k+N-n$ except for Section~\ref{section:Super_case}.

\section{Rectangular $W$-algebras}\label{section:Rectangular_W-algebras}

The $W$-algebra $\mathcal{W}^{\kappa}(\frg,f)$ is defined by the quantized Drinfeld-Sokolov reduction from the data consisting of a reductive Lie algebra $\frg$ over $\bbC$ equipped with a $\frg$-invariant symmetric bilinear form $\kappa$ 
and a nilpotent element $f$ in $\frg$.
In this paper, we do not use its definition but a characterization by the Miura map, which will be recalled in Section~\ref{section:Miura_map}.
See \cite{MR2013802} for the definition.

Let $\frg = \gl_{N}$ be the complex general linear Lie algebra consisting of $N \times N$ matrices.
We denote by $E_{i,j}$ the matrix unit whose $(i,j)$-entry is $1$.
Put $I_{N} = \sum_{i=1}^{N} E_{i,i}$ and $\frz_{N} = \bbC I_{N}$.
Then we have a decomposition
\[
	\frg = \fraksl_{N} \oplus \frz_{N}.
\]
Fix $k \in \bbC$.
We define an invariant symmetric bilinear form $\kappa$ on $\frg$ by
\begin{equation}
	\kappa(X,Y) = \begin{cases}
		k \tr(XY) & \text{if } X,Y \in \fraksl_{N},\\
		(k+N) \tr(XY) & \text{if } X,Y \in \frz_{N},\\
		0 & \text{if } X \in \fraksl_{N} \text{ and } Y \in \frz_{N}.
	\end{cases} \label{eq:kappa}
\end{equation}
The explicit value of $\kappa$ is
\[
	\kappa(E_{i,j},E_{p,q}) = \delta_{iq}\delta_{jp}k + \delta_{ij}\delta_{pq}.
\]

We assume $N= l n$ with nonzero positive integers $l$ and $n$.
As vector spaces, we have an isomorphism
\begin{equation}
	\gl_{l} \otimes \gl_{n} \cong \frg=\gl_{N}, \qquad E_{s,t} \otimes E_{i,j} \mapsto E_{(s-1)n+i,(t-1)n+j}. \label{eq:isom_gl} 
\end{equation}
We define a $\bbZ$-grading of $\gl_{l}$ by $\deg E_{s,t} = t-s$ and denote by $(\gl_{l})_{j}$ its degree $j$ component.
It induces a $\bbZ$-grading of $\frg$ by
\[
	\frg_{j} = (\gl_{l})_{j} \otimes \gl_{n}
\]
using the identification (\ref{eq:isom_gl}).
In particular, we have $\frg_{0} \cong \gl_{n}^{\oplus l}$.
We set
\begin{equation}
	E_{i,j}^{[s]} = E_{(s-1)n+i,(s-1)n+j} \label{eq:diagonal}
\end{equation}
for $s=1,\ldots,l$ and $i,j=1,\ldots,n$, then $\{ E_{i,j}^{[s]} \mid i,j=1,\ldots,n \}$ forms a $\bbC$-basis of the $s$-th component of $\frg_{0} \cong \gl_{n}^{\oplus l}$. 
Set
\begin{equation}
	\begin{split}
		f = \left( \sum_{s=1}^{l-1} E_{s+1,s} \right) \otimes I_{n} = \sum_{s=1}^{l-1} \sum_{i=1}^{n} E_{sn+i,(s-1)n+i} = \left[ \begin{array}{ccccc}
			O &  & & & O\\ 
			I_{n} & O & & & \\ 
			 & I_{n} & \dots & & \\ 
			 &  & & O & \\ 
			O &  & &  I_{n} & O
		\end{array} \right]. \label{eq:nilpotent}
	\end{split}
\end{equation}
The Jordan form of the nilpotent element $f$ corresponds to the partition $(l^{n})$ of rectangular type.

From the data ($\frg, \kappa, f$) above, we can define the rectangular $W$-algebra $\rect$ as a vertex algebra.
It is isomorphic to the tensor product
\[
	\mathcal{W}^{k}(\fraksl_{N},f) \otimes V^{k+N}(\frz_{N}),
\]
where $\mathcal{W}^{k}(\fraksl_{N},f)$ is the $W$-algebra associated with $\fraksl_{N}$ and $f$ at level $k$, 
and $V^{k+N}(\frz_{N})$ is the Heisenberg vertex algebra associated with $\frz_{N}$ at level $k+N$.
 
\section{Affine Lie algebras}\label{section:Affine_Lie_algebras}

Let $\kappa_{\frg}^{\circ}$ and $\kappa_{\frg_{0}}^{\circ}$ be the Killing forms on $\frg$ and $\frg_{0}$ respectively.
We define an invariant symmetric bilinear form $\tilde{\kappa}$ on $\frg_{0}$ by
\[
	\tilde{\kappa} = \kappa |_{\frg_{0}} + \dfrac{1}{2} ( \kappa_{\frg}^{\circ} |_{\frg_{0}} - \kappa_{\frg_{0}}^{\circ} ).
\]
This bilinear form is often denoted by $\tau_{\kappa}$, but we do not use the notation in this paper.
The explicit value of $\tilde{\kappa}$ is
\begin{equation}
	\tilde{\kappa}(E_{i,j}^{[s]},E_{p,q}^{[t]}) = \delta_{st} \Big( \delta_{iq}\delta_{jp}(k+(l-1)n) + \delta_{ij}\delta_{pq} \Big). \label{eq:tildekappa}
\end{equation}

For an arbitrary $\alpha \in \bbC$, we can define an invariant symmetric bilinear form $(\, , \,)$ on $\gl_{n}$ by replacing $k$ and $N$ with $\alpha$ and $n$ in the definition of $\kappa$ in (\ref{eq:kappa});
\[
	(X,Y) = \begin{cases}
		\alpha \tr(XY) & \text{if } X,Y \in \fraksl_{n},\\
		(\alpha+n) \tr(XY) & \text{if } X,Y \in \frz_{n},\\
		0 & \text{if } X \in \fraksl_{n} \text{ and } Y \in \frz_{n}.
	\end{cases}
\]
Its explicit value is
\begin{equation}
	(E_{i,j},E_{p,q}) = \delta_{iq}\delta_{jp}\alpha + \delta_{ij}\delta_{pq}. \label{eq:gln_bilinear_form}
\end{equation}
Hence, if we put
\[
	\alpha = k+(l-1)n
\]
and compare (\ref{eq:tildekappa}) and (\ref{eq:gln_bilinear_form}), then the restriction of $\tilde{\kappa}$ to each component of $\frg_{0} \cong \gl_{n}^{\oplus l}$ is identified with $(\, , \,)$.
We use the same symbol $\tilde{\kappa}$ for this bilinear form on $\gl_{n}$ by abuse of notation.

We define the affine Lie algebra $\hat{\frg}_{0}$ associated with $\tilde{\kappa}$ by
\[
	\hat{\frg}_{0} = \frg_{0}[t,t^{-1}] \oplus \bbC \mathbf{1},
\]
\[
	[X t^{m}, Y t^{m'}] = [X,Y] t^{m+m'} + m \delta_{m+m',0} \tilde{\kappa}(X,Y) \mathbf{1}, \quad \text{$\mathbf{1}$ is central}.
\]
Similarly we define the affine Lie algebra $\hat{\gl}_{n}$ associated with $\tilde{\kappa}$ by
\[
	\hat{\gl}_{n} = \gl_{n}[t,t^{-1}] \oplus \bbC \mathbf{1}',
\]
\[
	[X t^{m}, Y t^{m'}] = [X,Y] t^{m+m'} + m \delta_{m+m',0} \tilde{\kappa}(X,Y) \mathbf{1}', \quad \text{$\mathbf{1}'$ is central}.
\]
We have an isomorphism
\[
	\hat{\frg}_{0} / (\mathbf{1}-1) \cong \left( \hat{\gl}_{n} / (\mathbf{1}'-1) \right)^{\oplus l}.
\]
We put
\[
	\hat{\gl}_{n}^{\alpha} = \hat{\gl}_{n} / (\mathbf{1}'-1).
\] 
It is the direct sum of the affine Lie algebra $\hat{\fraksl}_{n}$ at level $\alpha$ and the Heisenberg Lie algebra $\hat{\frz}_{n}$ at level $\alpha+n$.

We consider the universal affine vertex algebras
\[
	V^{\tilde{\kappa}}(\frg_{0}) = U(\hat{\frg}_{0}) \otimes_{U(\frg_{0}[t] \oplus \bbC \mathbf{1})} \bbC, \quad V^{\tilde{\kappa}}(\gl_{n}) = U(\hat{\gl}_{n}) \otimes_{U(\gl_{n}[t] \oplus \bbC \mathbf{1}')} \bbC,
\]
where $\bbC$ denotes both the one-dimensional representations defined by
\[
	\frg_{0}[t], \gl_{n}[t] \mapsto 0, \quad \mathbf{1}, \mathbf{1}' \mapsto 1.
\]
We have an isomorphism
\[
	V^{\tilde{\kappa}}(\frg_{0}) \cong V^{\tilde{\kappa}}(\gl_{n})^{\otimes l}
\]
of vertex algebras.
We also have an isomorphism
\[
	V^{\tilde{\kappa}}(\gl_{n}) \cong V^{\alpha}(\fraksl_{n}) \otimes V^{\alpha+n}(\frz_{n}),
\]
where $V^{\alpha}(\fraksl_{n})$ is the universal affine vertex algebra associated with $\hat{\fraksl}_{n}$ at level $\alpha$, and $V^{\alpha+n}(\frz_{n})$ is the Heisenberg vertex algebra associated with $\hat{\frz}_{n}$ at level $\alpha+n$.

We denote the element $X t^{m}$ by $X(m)$.
As vector spaces, we have
\[
	V^{\tilde{\kappa}}(\frg_{0}) \cong U(\frg_{0}[t^{-1}]t^{-1}) \ket{0}, 
\]
where $\ket{0} = 1 \otimes 1$ denotes the vacuum vector.
For $X \in \frg_{0}$ and $m \in \bbZ_{\geq 1}$, we use the symbol $X[-m]$ for the element $X(-m)\ket{0}$ of $V^{\tilde{\kappa}}(\frg_{0})$.

We associate with a vertex algebra $V$ the current algebra $\mathfrak{U}(V)$ (sometimes called the universal enveloping algebra).
Roughly speaking, it is a completed associative algebra topologically generated by elements $vt^{m}$ ($v \in V$ and $m \in \bbZ$) corresponding to the Fourier coefficients $v_{(m)}$ of the generating field $v(z)=\sum_{m \in \bbZ} v_{(m)}z^{-m-1}$.
See \cite[Chapter~4]{MR2082709} or \cite[Section~3.11]{MR2318558} for details (it is denoted by $\widetilde{U}(V)$ in \cite{MR2082709}).

Consider the universal affine vertex algebra $V^{\tilde{\kappa}}(\gl_{n})$.
Its current algebra $\mathfrak{U}(V^{\tilde{\kappa}}(\gl_{n}))$ is identified with a completion $U(\hat{\gl}_{n}^{\alpha})_{\mathrm{comp}}$ of $U(\hat{\gl}_{n}^{\alpha})$ given as follows.
In general, let $A = \bigoplus_{d \in \bbZ} A_d$ be a $\bbZ$-graded algebra.
We call
\[
	A_{\mathrm{comp}} = \bigoplus_{d \in \bbZ} \lim_{\substack{\leftarrow\\ m}} \left( A_d / \sum_{r > m} A_{d-r} A_{r} \right)
\]
the standard degree-wise completion of $A$ according to \cite[A.2]{MR2318558}.
We define a grading of $U(\hat{\gl}_{n}^{\alpha})$ by $\deg X(m) = m$.
Then we denote by $U(\hat{\gl}_{n}^{\alpha})_{\mathrm{comp}}$ the standard degree-wise completion of $U(\hat{\gl}_{n}^{\alpha})$.
Moreover $U(\hat{\gl}_{n}^{\alpha})_{\mathrm{comp}}^{\otimes l}$ denotes the degree-wise completed tensor product of $U(\hat{\gl}_{n}^{\alpha})$ defined in the same way.
Then we have
\begin{equation}
	\mathfrak{U}(V^{\tilde{\kappa}}(\frg_{0})) \cong U(\hat{\gl}_{n}^{\alpha})_{\mathrm{comp}}^{\otimes l}. \label{eq:tensor}
\end{equation}
Let the symbol $X^{[s]}$ for $s=1,\ldots,l$ denote the element $1^{\otimes (s-1)} \otimes X \otimes 1^{\otimes (l-s)}$.
This notation is compatible with one in (\ref{eq:diagonal}) under the identification (\ref{eq:tensor}).

\section{Miura map and Arakawa-Molev's generators}\label{section:Miura_map}

The Miura map is an injective homomorphism
\begin{equation}
	\rect \to V^{\tilde{\kappa}}(\frg_{0}) \cong V^{\tilde{\kappa}}(\gl_{n})^{\otimes l} \label{eq:Miura}
\end{equation}
of vertex algebras.
See \cite[Section~3]{MR3598875} for details.
Consider the current algebras of (\ref{eq:Miura}), then we have an injective homomorphism
\begin{equation}
	\mathfrak{U}(\rect) \to U(\hat{\gl}_{n}^{\alpha})_{\mathrm{comp}}^{\otimes l} \label{eq:Miura_algebra}
\end{equation}
of associative algebras.
We also call (\ref{eq:Miura_algebra}) Miura map.
In the sequel, we identify $\rect$ with its image by the Miura map and regard it as a subalgebra of $V^{\tilde{\kappa}}(\gl_{n})^{\otimes l}$.
Also we regard $\mathfrak{U}(\rect)$ as a subalgebra of $U(\hat{\gl}_{n}^{\alpha})_{\mathrm{comp}}^{\otimes l}$.

We recall generators of $\rect$ introduced by Arakawa-Molev~\cite{MR3598875}.
Let $A^{[s]} = \left( E_{i,j}^{[s]}[-1]_{(-1)} \right)_{1 \leq i,j \leq n}$ be the $n \times n$ matrix whose $(i,j)$-entry is $E_{i,j}^{[s]}[-1]_{(-1)}$.
Let $\tau$ be a formal operator satisfying
\[
	[X[-m], \tau] = \alpha m X[-m-1]
\]
for $X \in \frg_{0}$.
We define a matrix $W^{(r)}$ by
\begin{equation}
	(\tau + A^{[1]}) (\tau + A^{[2]}) \cdots (\tau + A^{[l]}) = \tau^{l} + \sum_{r=1}^{l} \tau^{l-r} W^{(r)}. \label{eq:Miura_transformation}
\end{equation}
In (\ref{eq:Miura_transformation}), if we obtain an element of the form $X_1[-m_1]_{(-1)} \cdots X_a[-m_a]_{(-1)}$, then we regard the right-most $X_a[-m_a]_{(-1)}$ as $X_a[-m_a]$.
Then the $(i,j)$-entry, denoted by $W_{i,j}^{(r)}$, of $W^{(r)}$ gives an element of $V^{\tilde{\kappa}}(\frg_{0})$.

\begin{thm}[\cite{MR3598875}, Theorem~3.1 and Corollary~3.2]
The elements $W_{i,j}^{(r)}$ belong to $\rect$ for all $r=1, \ldots, l$ and $i,j=1, \ldots, n$.
They freely generate $\rect$ as a vertex algebra.
\end{thm}

\begin{exam}
We have
\begin{align*}
	W_{i,j}^{(1)} &= \sum_{s=1}^{l} E_{i,j}^{[s]}[-1],\\
	W_{i,j}^{(2)} &= \sum_{1 \leq s_1 < s_2 \leq l} \sum_{a=1}^{n} E_{i,a}^{[s_1]}[-1]_{(-1)} E_{a,j}^{[s_2]}[-1] + \alpha \sum_{s=1}^{l} (l-s) E_{i,j}^{[s]}[-2].
\end{align*}
\end{exam}

\begin{rem}\label{rem:difference}
The elements $W_{i,j}^{(r)}$ are slightly different from ones in the literature.
\begin{enumerate}
\item
Let ${}^{\mathrm{AM}}W_{i,j}^{(r)}$ denote the image by the Miura map of the generators given in \cite[Corollary~3.2]{MR3598875}.
They are defined by
\begin{equation*}
	(- \tau + A^{[1]}) (- \tau + A^{[2]}) \cdots (- \tau + A^{[l]}) = (- \tau)^{l} + \sum_{r=1}^{l} \left( {}^{\mathrm{AM}}W_{j,i}^{(r)} \right)_{1 \leq i,j \leq n} (- \tau)^{l-r}.
\end{equation*}
We have an algebra automorphism $\sigma$ on $V^{\tilde{\kappa}}(\frg_{0})$ induced from $X^{[s]} \mapsto X^{[l-s+1]}$ for $X \in \gl_{n}$.
Under the identification
\[
	V^{\tilde{\kappa}}(\frg_{0}) \cong V^{\tilde{\kappa}}(\gl_{n})^{\otimes l},
\]
$\sigma$ corresponds to the operation which reverses the order of the tensor product.
Then we have
\[
	W_{i,j}^{(r)} = \sigma \left( {}^{\mathrm{AM}}W_{j,i}^{(r)} \right).
\]

\item
In \cite{MR4061286}, generators ${{U_{(r)}}^{a}}_{b}$ are used.
They are defined by
\begin{equation*}
	(\tau + A^{[1]}) (\tau + A^{[2]}) \cdots (\tau + A^{[l]}) = \tau^{l} + \sum_{r=1}^{l} \left( {{U_{(r)}}^{a}}_{b} \right)_{1 \leq a,b \leq n} \tau^{l-r}.
\end{equation*}
The relation with our generators $W_{i,j}^{(r)}$ for $r=1,2$ is given as follows:
\[
	W_{i,j}^{(1)} = {{U_{(1)}}^{i}}_{j}, \quad W_{i,j}^{(2)} = {{U_{(2)}}^{i}}_{j} + (l-1)\alpha \partial {{U_{(1)}}^{i}}_{j}.
\]
Formulas for OPE appearing in this paper can be obtained from ones given in \cite[Section~3.1]{MR4061286}.
\end{enumerate}
\end{rem}

The current algebra $\mathfrak{U}(\rect)$ is an associative algebra which has topological generators
\[
	\{ W_{i,j}^{(r)}t^{m} \mid 1 \leq r \leq l, \ \ 1 \leq i,j \leq n, \ \ m \in \bbZ \},
\]
where the element $W_{i,j}^{(r)}t^{m}$ corresponds to the Fourier coefficient $(W_{i,j}^{(r)})_{(m)}$ of the element $W_{i,j}^{(r)}$ of $\rect$.
We set
\[
	W_{i,j}^{(r)}(m) = W_{i,j}^{(r)}t^{m+r-1}.
\] 

\begin{exam}\label{exam:generators}
We have
\begin{align*}
	W_{i,j}^{(1)}(m) &= \sum_{s=1}^{l} E_{i,j}^{[s]}(m),\\
	W_{i,j}^{(2)}(m) &= \sum_{r \in \bbZ} \sum_{1 \leq s_1 < s_2 \leq l} \sum_{a=1}^{n} E_{i,a}^{[s_1]}(m+r) E_{a,j}^{[s_2]}(-r) - (m+1)\alpha \sum_{s=1}^{l} (l-s) E_{i,j}^{[s]}(m).
\end{align*}
\end{exam}

The elements $W_{i,j}^{(1)}(m)$ satisfy the following commutation relation:
\[
	[W_{i,j}^{(1)}(m), W_{p,q}^{(1)}(m')] = \delta_{pj} W_{i,q}^{(1)}(m+m') - \delta_{iq} W_{p,j}^{(1)}(m+m') + m \delta_{m+m',0} l (\delta_{iq}\delta_{jp}\alpha + \delta_{ij}\delta_{pq}).
\]
That is, they generate the affine Lie algebra of $\mathfrak{gl}_{n}$ associated with the bilinear form $l \times \tilde{\kappa}$.

The set $\{ W_{i,j}^{(r)} \mid 1 \leq r \leq l, \ 1 \leq i,j \leq n \}$ forms a free generator of $\rect$.
One can reduce the generators to the elements for $r=1,2$ under the assumption $n \geq 2$ and $\alpha \neq 0$, although they do not freely generate $\rect$. 

\begin{prop}[\cite{Affine_super_Yangians_and_rectangular_W-superalgebras}]\label{prop:generators}
Assume $n \geq 2$ and $\alpha \neq 0$.
Then the elements $W_{i,j}^{(r)}$ for $r=1,2$ and all $i,j$ generate $\rect$ as a vertex algebra.
In particular, $W_{i,j}^{(r)}(m)$ for $r=1,2$ and all $i,j,m$ topologically generate $\mathfrak{U}(\rect)$ as an associative algebra.
\end{prop}

\section{Parabolic induction}\label{section:Parabolic_induction}

Let us recall the parabolic induction for the rectangular $W$-algebras (see \cite[Introduction,  Theorem~A and B]{MR4091897} for a general result).

Choose integers $l_1, l_2 \geq 1$ satisfying $l=l_1+l_2$.
Set $N_1 = l_1 n$ and $N_2 = l_2 n$ so that $N=N_1+N_2$.
We also take $k_1$ and $k_2$ so that
\begin{equation}
	k+N = k_1+N_1 = k_2+N_2 \label{eq:relations_among_levels}
\end{equation}
holds.
Since the parameter $\alpha$ is defined by $k+N = \alpha + n$, the equality (\ref{eq:relations_among_levels}) means that $\alpha$ is unchanged if we replace $(k,l,N)$ with $(k_i,l_i,N_i)$ ($i=1,2$).

For $(k_i,l_i,N_i)$ ($i=1,2$), we define an invariant symmetric bilinear form $\kappa_{i}$ on $\gl_{N_i}$ in a way similar to (\ref{eq:kappa}), and a nilpotent element $f_{i}$ as in (\ref{eq:nilpotent}) corresponding to the partition ($l_i^{n}$).
Associated with the data, $\mathcal{W}^{\kappa_1}(\gl_{N_1},f_1)$ and $\mathcal{W}^{\kappa_2}(\gl_{N_2},f_2)$ are defined.
We have the Miura maps
\[
	\mathcal{W}^{\kappa_i}(\gl_{N_i},f_i) \to V^{\tilde{\kappa}_i}(\gl_{n})^{\otimes l_i}
\]
for $i=1,2$.
Since $\alpha$ is taken uniformly for $(k,l,N)$ and $(k_i,l_i,N_i)$, $\tilde{\kappa}_i$ on $\gl_{n}$ is the same as $\tilde{\kappa}$.
Hence $\mathcal{W}^{\kappa_1}(\gl_{N_1},f_1) \otimes \mathcal{W}^{\kappa_2}(\gl_{N_2},f_2)$ can be regarded as a subalgebra of $V^{\tilde{\kappa}}(\gl_{n})^{\otimes l}$ by the tensor product of the Miura maps.
Then it turns out that $\mathcal{W}^{\kappa_1}(\gl_{N_1},f_1) \otimes \mathcal{W}^{\kappa_2}(\gl_{N_2},f_2)$ contains $\rect$ inside $V^{\tilde{\kappa}}(\gl_{n})^{\otimes l}$.  
We denote the inclusion by $\Delta_{l_1,l_2}$:
\[
	\Delta_{l_1,l_2} \colon \rect \to \mathcal{W}^{\kappa_1}(\gl_{N_1},f_1) \otimes \mathcal{W}^{\kappa_2}(\gl_{N_2},f_2).
\]

The generators given in Section~\ref{section:Miura_map} are related by $\Delta_{l_1,l_2}$ as follows.
Divide the left-hand side $(\tau + A^{[1]}) \cdots (\tau + A^{[l]})$ of (\ref{eq:Miura_transformation}) into two blocks:
\[
	(\tau + A^{[1]}) \cdots (\tau + A^{[l_1]}) \ \text{ and } \ (\tau + A^{[l_1+1]}) \cdots (\tau + A^{[l]}).
\]
Then we can similarly define generators for $\mathcal{W}^{\kappa_1}(\gl_{N_1},f_1)$ and $\mathcal{W}^{\kappa_2}(\gl_{N_2},f_2)$, and we can relate them to the generators $W_{i,j}^{(r)}$ of $\rect$.
This gives a description of $\Delta_{l_1,l_2}$ in terms of generators.
We give its explicit form for $r=1,2$ as follows:
\[
	\Delta_{l_1,l_2}(W_{i,j}^{(1)}(m)) = W_{i,j}^{(1)}(m) \otimes 1 + 1 \otimes W_{i,j}^{(1)}(m),
\]
\begin{equation*}
	\begin{split}
		\Delta_{l_1,l_2}(W_{i,j}^{(2)}(m)) &= W_{i,j}^{(2)}(m) \otimes 1 + 1 \otimes W_{i,j}^{(2)}(m) + \sum_{r \in \bbZ} \sum_{a=1}^{n} W_{i,a}^{(1)}(m+r) \otimes W_{a,j}^{(1)}(-r) \\
		&\qquad - (m+1) l_2 \alpha W_{i,j}^{(1)}(m) \otimes 1.
	\end{split}
\end{equation*} 

We will relate $\Delta_{l_1,l_2}$ to the coproduct of the affine Yangian in Corollary~\ref{cor:parabolic_induction} and \ref{cor:parabolic_induction2} as a consequence of the main result of this paper.

\section{Affine Yangians}\label{section:Affine_Yangians}

Assume $n \geq 3$ until the end of Section~\ref{section:Commuting_elements}.
Let 
\[
	a_{ij} =
	\begin{cases}
		2  &\text{if } i=j, \\
		-1 &\text{if } i=j \pm 1 \mod n, \\
		0  &\text{otherwise}
	\end{cases}
\]
be entries of the Cartan matrix of type $A_{n-1}^{(1)}$.
We use the notation $\{x,y\}=xy+yx$.

\begin{dfn}\label{dfn:affine_Yangian}
The affine Yangian $\affY = Y_{\hbar,\ve}(\affsl)$ is the algebra over $\bbC$ generated by $X_{i,r}^{+}$, $X_{i,r}^{-}$, $H_{i,r}$ ($i = 0,1,\ldots,n-1$ and $r \in \mathbb{Z}_{\geq 0})$ with parameters $\hbar, \ve \in \bbC$ subject to the relations:
\begin{equation*}
	[H_{i,r},H_{j,s}]=0, \qquad [X_{i,r}^{+},X_{j,s}^{-}] = \delta_{ij} H_{i,r+s}, \qquad [H_{i,0},X_{j,s}^{\pm}]=\pm a_{ij} X_{j,s}^{\pm} \quad \text{ for any $i,j$,}
\end{equation*}
\begin{align*}
	&[H_{i, r+1}, X_{j, s}^{\pm}] - [H_{i, r}, X_{j, s+1}^{\pm}] 
	= \pm \dfrac{a_{ij}}{2} \hbar \{ H_{i, r}, X_{j, s}^{\pm} \}, \\
	&[X_{i, r+1}^{\pm}, X_{j, s}^{\pm}] - [X_{i, r}^{\pm}, X_{j, s+1}^{\pm}] 
 	= \pm \dfrac{a_{ij}}{2} \hbar \{ X_{i, r}^{\pm}, X_{j, s}^{\pm} \} \quad \text{for $(i,j) \neq (0,n-1), (n-1,0)$,}
\end{align*}
\begin{equation*}
	[H_{n-1, r+1}, X_{0, s}^{\pm}] - [H_{n-1, r}, X_{0, s+1}^{\pm}] 
	=\mp\dfrac{\hbar}{2}\{H_{n-1, r}, X_{0, s}^{\pm}\}+\left( \dfrac{n}{2}\hbar + \ve \right)[H_{n-1, r}, X_{0, s}^{\pm}], 
\end{equation*}
\begin{equation*}
	[H_{0, r+1}, X_{n-1, s}^{\pm}] - [H_{0, r}, X_{n-1, s+1}^{\pm}] 
	=\mp\dfrac{\hbar}{2}\{H_{0, r}, X_{n-1, s}^{\pm}\}-\left( \dfrac{n}{2}\hbar + \ve \right)[H_{0, r}, X_{n-1, s}^{\pm}], 
\end{equation*}
\begin{equation*}
	[X_{n-1, r+1}^{\pm}, X_{0, s}^{\pm}] - [X_{n-1, r}^{\pm}, X_{0, s+1}^{\pm}] 
	=\mp\dfrac{\hbar}{2}\{X_{n-1, r}^{\pm}, X_{0, s}^{\pm}\}+\left( \dfrac{n}{2}\hbar + \ve \right)[X_{n-1, r}^{\pm}, X_{0, s}^{\pm}], 
\end{equation*}
\begin{equation*}
	\sum_{w \in \frakS_{1-a_{ij}}} [X_{i,r_{w(1)}}^{\pm}, [X_{i,r_{w(2)}}^{\pm}, \dots, [X_{i,r_{w(1 - a_{ij})}}^{\pm}, X_{j,s}^{\pm}]\dots]] = 0 \quad \text{for $i \neq j$.}
\end{equation*}
\end{dfn}

Set
\[
	X_{i}^{\pm}(z) = \hbar \sum_{r \geq 0} X_{i,r}^{\pm} z^{-r-1}, \quad H_{i}(z) = 1 + \hbar \sum_{r \geq 0} H_{i,r} z^{-r-1}.
\]

\begin{rem}
\begin{enumerate}
\item
The defining relations given in Definition~\ref{dfn:affine_Yangian} is different from those in \cite{MR4207398}.
Generators in \cite{MR4207398} are denoted by $x_{i,r}^{+}$, $x_{i,r}^{-}$, $h_{i,r}$ ($i = 0,1,\ldots,n-1$ and $r \in \mathbb{Z}_{\geq 0})$ with parameters $\ve_1, \ve_2$.
We set $\hbar = \ve_1 + \ve_2$ and
\[
	x_{i}^{\pm}(z) = \hbar \sum_{r \geq 0} x_{i,r}^{\pm} z^{-r-1}, \quad h_{i}(z) = 1 + \hbar \sum_{r \geq 0} h_{i,r} z^{-r-1}.
\]
Then the isomorphism is given by
\[
	X_{i}^{\pm}(z) = x_{i}^{\pm}\left(z-\dfrac{i}{2}( \ve_1-\ve_2 )\right), \quad
	H_{i}(z) = h_{i}\left(z-\dfrac{i}{2}( \ve_1-\ve_2 )\right)
\]
for $i=0,1,\ldots,n-1$ and
\[
	\hbar = \ve_1 + \ve_2, \quad \ve = -n\ve_2.
\]
The condition among the parameters and the central element for the evaluation map in \cite[Theorem~3.8]{MR4207398} is $\hbar c = - n \ve_1$.
This is equivalent to $\hbar(c + n) = - \ve$ in terms of the parameters in this paper. 

\item
The presentation in Definition~\ref{dfn:affine_Yangian} is given by Feigin-Finkelberg-Negut-Rybnikov \cite[3.15]{MR2827177}.
They use parameters $\hbar$, $\hbar'$.
The relation to our parameters is:
\[
	\text{$\hbar$ is the same}, \quad \hbar' = -n\ve_1 = -(n\hbar + \ve).
\]
\end{enumerate}
\end{rem}

The following is an immediate consequence of the defining relations.
\begin{prop}\label{prop:degree_one_generator}
The affine Yangian $\affY$ is generated by $X_{i,0}^{+}$, $X_{i,0}^{-}$, and $X_{i,1}^{+}$ for $i=0,1,\ldots,n-1$.
\end{prop}

We define a grading of $\affY$ by setting the degree of $X_{0,r}^{\pm}$ to be $\pm 1$ and that of other generators to be $0$.
Then we denote by $\compaffY$ the standard degree-wise completion of $\affY$.
Moreover $\compaffY^{\otimes l}$ denotes the degree-wise completed tensor product defined in the same way.
The target space of the coproduct will be $\compaffY^{\otimes 2}$.

\section{Coproduct}\label{section:Coproduct}

We introduce the coproduct for $\affY$ according to \cite{MR3861718}. 
In \cite{MR3861718}, Guay-Nakajima-Wendlandt constructed the coproduct for affine Yangians, whose explicit form was first discovered by Guay~\cite{MR2323534} for type A.

The affine Lie algebra $\affsl$ is defined by
\[
	\hat{\fraksl}_{n} = \fraksl_{n}[t,t^{-1}] \oplus \bbC c,
\]
\[
	[X t^{m}, Y t^{m'}] = [X,Y] t^{m+m'} + m \delta_{m+m',0} \tr(XY) c, \quad \text{$c$ is central}.
\]
We prepare some notations for $\affsl$.
We denote the element $X t^{m}$ by $X(m)$.
Set
\begin{gather*}
	X_{0}^{+} = E_{n,1}(1), \quad X_{0}^{-} = E_{1,n}(-1), \quad H_{0} = E_{n,n} - E_{1,1} + c,\\
	X_{i}^{+} = E_{i,i+1}, \quad X_{i}^{-} = E_{i+1,i}, \quad H_{i} = E_{i,i}-E_{i+1,i+1} \ \ (i \neq 0).
\end{gather*}
Let $\hat{\Delta}_{+}$ and $\hat{\Delta}_{+}^{\mathrm{re}}$ be the set of positive roots and of positive real roots.
Let $\mult \gamma$ denote the multiplicity of a root $\gamma$.
We fix root vectors $X_{\pm \gamma}^{(r)}$ for each $\gamma \in \hat{\Delta}_{+}$ and $r=1,\ldots,\mult \gamma$ satisfying $(X_{\gamma}^{(r)},X_{-\gamma}^{(s)})=\delta_{rs}$ with respect to the standard invariant symmetric bilinear form.
The simple roots are denoted by $\{ \alpha_i \mid i=0,1,\ldots,n-1\}$.
 
The subalgebra of $\affY$ generated by $X_{i,0}^{+}$, $X_{i,0}^{-}$, $H_{i,0}$ ($i = 0,1,\ldots,n-1$) is isomorphic to $U(\affsl)$ (see \cite[Theorem~6.1]{MR2323534} for $n \geq 4$ and \cite[Theorem~6.9]{MR4014633} in general).
We identify $U(\affsl)$ with the subalgebra and use the same symbols for corresponding elements of both $U(\affsl)$ and $\affY$.

We use the symbol $\square$ for $\square(X) = X \otimes 1 + 1 \otimes X$.

\begin{thm}[\cite{MR2323534}, Section~6, pp.\ 462, \cite{MR3861718}, Definition~4.6, Theorem~4.9, Proposition~5.18, Section~7, Proposition~4.24]
There exists an algebra homomorphism $\Delta \colon \affY \to \compaffY^{\otimes 2}$ uniquely determined by
\[
	\Delta(X_{i,0}^{\pm}) = \square(X_{i,0}^{\pm}),\quad \Delta(H_{i,0}) = \square(H_{i,0}),
\]
\begin{align*}
	\Delta(X_{i,1}^{+}) &= \square(X_{i,1}^{+}) + \hbar \Bigg( X_{i}^{+} \otimes H_{i} - \sum_{\gamma \in \hat{\Delta}_{+}} \sum_{r=1}^{\mult \gamma} [X_{i}^{+}, X_{\gamma}^{(r)}] \otimes X_{-\gamma}^{(r)} \Bigg),\\
	\Delta(X_{i,1}^{-}) &= \square(X_{i,1}^{-}) + \hbar \Bigg( H_{i} \otimes X_{i}^{-} + \sum_{\gamma \in \hat{\Delta}_{+}} \sum_{r=1}^{\mult \gamma} X_{\gamma}^{(r)} \otimes [X_{i}^{-}, X_{-\gamma}^{(r)}] \Bigg),\\
	\Delta(H_{i,1}) &= \square(H_{i,1}) + \hbar \Bigg( H_{i} \otimes H_{i} - \sum_{\gamma \in \hat{\Delta}_{+}^{\mathrm{re}}} (\alpha_i,\gamma)\, X_{\gamma}^{(1)} \otimes X_{-\gamma}^{(1)} \Bigg) 
\end{align*}
for $i=0,1,\ldots,n-1$.
Moreover, it satisfies the coassociativity $(\Delta \otimes \id) \circ \Delta = (\id \otimes \Delta) \circ \Delta$.
\end{thm}

We remark that the formulation is different from one in $\cite{MR3861718}$ for two points.
First, in $\cite{MR3861718}$, a completion is defined for Yangians associated with symmetrizable Kac-Moody Lie algebras while our completion is specific to affine type.
We need it as we will consider the composition of the coproduct and the evaluation map, and compare its image with the rectangular $W$-algebra.
Second, the coproduct in this paper is opposite to that of \cite{MR2323534} and \cite{MR3861718}.
Namely the order of the tensor product is reversed. 

An explicit form of $\Delta$ is derived by a straightforward computation as follows.

\begin{prop}\label{prop:explicit_coproduct}
We have
\begin{align*}
	&\Delta(X_{0,1}^{+}) = \square(X_{0,1}^{+}) \\
	&\quad + \hbar \Bigg( X_{0}^{+} \otimes c + \sum_{m \geq 0} \sum_{a=1}^{n} \Big( E_{a,1}(m+1) \otimes E_{n,a}(-m) - E_{n,a}(m+1) \otimes E_{a,1}(-m) \Big)\Bigg),\\
	&\Delta(X_{0,1}^{-}) = \square(X_{0,1}^{-}) \\
	&\quad + \hbar \Bigg( c \otimes X_{0}^{-} + \sum_{m \geq 0} \sum_{a=1}^{n} \Big( E_{a,n}(m) \otimes E_{1,a}(-m-1) - E_{1,a}(m) \otimes E_{a,n}(-m-1) \Big)\Bigg)
\end{align*}
and
\begin{align*}
	&\Delta(X_{i,1}^{+}) = \square(X_{i,1}^{+})\\
	&\quad + \hbar \sum_{m \geq 0} \Bigg( \sum_{a=1}^{i} E_{a,i+1}(m) \otimes E_{i,a}(-m) + \sum_{a=i+1}^{n} E_{a,i+1}(m+1) \otimes E_{i,a}(-m-1) \\
	&\qquad\qquad\qquad - \sum_{a=1}^{i} E_{i,a}(m+1) \otimes E_{a,i+1}(-m-1) - \sum_{a=i+1}^{n} E_{i,a}(m) \otimes E_{a,i+1}(-m) \Bigg),\\
	&\Delta(X_{i,1}^{-}) = \square(X_{i,1}^{-})\\
	&\quad + \hbar \sum_{m \geq 0} \Bigg( \sum_{a=1}^{i} E_{a,i}(m) \otimes E_{i+1,a}(-m) + \sum_{a=i+1}^{n} E_{a,i}(m+1) \otimes E_{i+1,a}(-m-1) \\
	&\qquad\qquad\qquad - \sum_{a=1}^{i} E_{i+1,a}(m+1) \otimes E_{a,i}(-m-1) - \sum_{a=i+1}^{n} E_{i+1,a}(m) \otimes E_{a,i}(-m) \Bigg)
\end{align*}
for $i=1,\ldots,n-1$.
\end{prop}
We do not use an explicit form of $\Delta(H_{i,1})$.

\section{Evaluation map}\label{section:Evaluation_map}

The evaluation map for $\affY$ was discovered by Guay in \cite{MR2323534}.
It requires a certain condition among the parameters as remarked in \cite{MR4207398}.
We state it in the form for the presentation of this paper.

\begin{rem}\label{rem:error}
The published version of \cite{MR4207398} has an error and the correction is available.
See also the arXiv version of \cite{MR4207398}, or Section~\ref{section:Commuting_elements} and Appendix~\ref{section:proof_of_evaluation_map} of this paper. 
\end{rem}

Recall that we set $\alpha = k + (l-1)n$ and defined $\hat{\gl}_{n}^{\alpha}$ as the direct sum of $\hat{\mathfrak{sl}}_{n}$ at level $\alpha$ and the Heisenberg Lie algebra at level $\alpha+n$ in Section~\ref{section:Affine_Lie_algebras}.
Although Theorem~\ref{thm:evaluation} below holds for an arbitrary $\alpha$ satisfying $\alpha + n = -\ve/\hbar$, we use it with $\alpha = k + (l-1)n$ for our purpose.

\begin{thm}[\cite{MR2323534}, Section~6, pp.\ 462--463, \cite{MR4207398}, Theorem~3.8]\label{thm:evaluation}
Assume $\alpha + n = -\ve/\hbar$.
Then there exists an algebra homomorphism $\ev \colon \affY \to U(\hat{\gl}_{n}^{\alpha})_{{\mathrm{comp}}}$ uniquely determined by 
\[
	\ev(X_{i,0}^{\pm})=X_{i}^{\pm}, \quad \ev(H_{i,0})=H_{i}, \quad \ev(c) = \alpha, 
\]
\begin{align*}
		\ev(X_{0,1}^{+}) &= \hbar \Bigg( \alpha X_{0}^{+} + \sum_{m \geq 0} \sum_{a=1}^{n} E_{n,a}(-m) E_{a,1}(m+1) \Bigg),\\
		\ev(X_{0,1}^{-}) &= \hbar \Bigg( \alpha X_{0}^{-} + \sum_{m \geq 0} \sum_{a=1}^{n} E_{1,a}(-m-1) E_{a,n}(m) \Bigg),\\
		\ev(H_{0,1}) &= \hbar \Bigg( \alpha H_{0} - E_{n,n} (E_{1,1}-\alpha) + \sum_{m \geq 0} \sum_{a=1}^{n} \Big( E_{n,a}(-m) E_{a,n}(m) - E_{1,a}(-m-1) E_{a,1}(m+1) \Big)\Bigg)
\end{align*}
and
\begin{align*}
		\ev(X_{i,1}^{+}) &= \hbar \Bigg( -\frac{i}{2} X_{i}^{+} + \sum_{m \geq 0} \Bigg( \sum_{a=1}^{i} E_{i,a}(-m) E_{a,i+1}(m) + \sum_{a=i+1}^{n} E_{i,a}(-m-1) E_{a,i+1}(m+1) \Bigg) \Bigg),\\
		\ev(X_{i,1}^{-}) &= \hbar \Bigg( -\frac{i}{2} X_{i}^{-} + \sum_{m \geq 0} \Bigg( \sum_{a=1}^{i} E_{i+1,a}(-m) E_{a,i}(m) + \sum_{a=i+1}^{n} E_{i+1,a}(-m-1) E_{a,i}(m+1) \Bigg) \Bigg),\\
		\ev(H_{i,1}) &= \hbar \Bigg( -\frac{i}{2} H_{i} - E_{i,i} E_{i+1,i+1} + \sum_{m \geq 0} \Bigg( \sum_{a=1}^{i} E_{i,a}(-m) E_{a,i}(m) + \sum_{a=i+1}^{n} E_{i,a}(-m-1) E_{a,i}(m+1) \\
		&- \sum_{a=1}^{i} E_{i+1,a}(-m) E_{a,i+1}(m) - \displaystyle\sum_{a=i+1}^{n} E_{i+1,a}(-m-1) E_{a,i+1}(m+1) \Bigg) \Bigg)
\end{align*}
for $i=1,\ldots,n-1$.
\end{thm}

\section{Affine Yangians and rectangular $W$-algebras}\label{section:Affine_Yangians_and_rectangular_W-algebras}

For each $\beta \in \bbC$, we define an algebra automorphism $\eta_{\beta}$ of $U(\hat{\gl}_{n}^{\alpha})$ by
\[
	E_{i,j}(m) \mapsto E_{i,j}(m) + \delta_{m,0}\delta_{i,j}\beta.
\]
We define an algebra automorphism $\eta^{(l)}$ of $U(\hat{\gl}_{n}^{\alpha})_{{\mathrm{comp}}}^{\otimes l}$ by
\[
	\eta^{(l)} = \eta_{(l-1)\alpha} \otimes \eta_{(l-2)\alpha} \otimes \cdots \otimes \eta_{\alpha} \otimes \id.
\]

We denote by $\Delta^{l-1}$ the $(l-1)$ times iterated composition of $\Delta$, which is a homomorphism from $\compaffY$ to $\compaffY^{\otimes l}$.
The map $\ev^{\otimes l} \colon \affY^{\otimes l} \to U(\hat{\gl}_{n}^{\alpha})_{{\mathrm{comp}}}^{\otimes l}$ is extended to an algebra homomorphism from $\compaffY^{\otimes l}$ to $U(\hat{\gl}_{n}^{\alpha})_{{\mathrm{comp}}}^{\otimes l}$.
We use the same letter $\ev^{\otimes l}$ for this map. 

\begin{dfn}\label{dfn:map}
Assume $k+N=\alpha+n=-\ve/\hbar$.
We set
\[
	\Phi_{l} = \eta^{(l)} \circ \ev^{\otimes l} \circ \Delta^{l-1},
\]
which is an algebra homomorphism from $\affY_{\mathrm{comp}}$ to $U(\hat{\gl}_{n}^{\alpha})_{{\mathrm{comp}}}^{\otimes l}$.
\end{dfn}
Let us state the main theorem.
\begin{thm}\label{thm:main}
The image of $\Phi_{l}$ is contained in $\mathfrak{U}(\rect)$.
That is, the map $\Phi_{l} = \eta^{(l)} \circ \ev^{\otimes l} \circ \Delta^{l-1}$ gives an algebra homomorphism
\[
	\Phi_{l} \colon \affY_{\mathrm{comp}} \to \mathfrak{U}(\rect).
\]
Moreover, it is surjective if $\alpha \neq 0$.
\end{thm}

\begin{proof}
Since we have
\begin{equation}
		\Phi_{l}(X_{i,0}^{+}) = \begin{cases} 
			W_{n,1}^{(1)}(1) & \text{ if $i=0$},\\
			W_{i,i+1}^{(1)}(0) & \text{ if $i \neq 0$},
		\end{cases}\quad
		\Phi_{l}(X_{i,0}^{-}) = \begin{cases}
			W_{1,n}^{(1)}(-1) & \text{ if $i=0$},\\
			W_{i+1,i}^{(1)}(0) & \text{ if $i \neq 0$},
		\end{cases}\label{eq:image_of_degree_zero}
\end{equation}
the image of $\{ X_{i,0}^{\pm} \mid 0 \leq i \leq n-1 \}$ is contained in $\mathfrak{U}(\rect)$.
Hence the following formulas conclude that the image of $\Phi_{l}$ is contained in $\mathfrak{U}(\rect)$ by Proposition~\ref{prop:degree_one_generator}. 
\begin{prop}\label{prop:image}
We have
\begin{equation*}
	\begin{split}
		\Phi_{l}(X_{0,1}^{+}) &= (-\hbar) \Bigg( W_{n,1}^{(2)}(1) - \alpha W_{n,1}^{(1)}(1) - \sum_{m \geq 0} \sum_{a=1}^{n} W_{n,a}^{(1)}(-m) W_{a,1}^{(1)}(m+1)\Bigg),\\
		\Phi_{l}(X_{0,1}^{-}) &= (-\hbar) \Bigg( W_{1,n}^{(2)}(-1) - l \alpha W_{1,n}^{(1)}(-1) - \sum_{m \geq 0} \sum_{a=1}^{n} W_{1,a}^{(1)}(-m-1) W_{a,n}^{(1)}(m) \Bigg)
	\end{split}
\end{equation*}
and
\begin{equation*}
	\begin{split}
		\Phi_{l}(X_{i,1}^{+}) &= (-\hbar) \Bigg( W_{i,i+1}^{(2)}(0) + \dfrac{i}{2} W_{i,i+1}^{(1)}(0) \\
		&\qquad - \sum_{m \geq 0} \Bigg( \sum_{a=1}^{i} W_{i,a}^{(1)}(-m) W_{a,i+1}^{(1)}(m) + \sum_{a=i+1}^{n} W_{i,a}^{(1)}(-m-1) W_{a,i+1}^{(1)}(m+1) \Bigg)\Bigg),\\
		\Phi_{l}(X_{i,1}^{-}) &= (-\hbar) \Bigg( W_{i+1,i}^{(2)}(0) + \dfrac{i}{2} W_{i+1,i}^{(1)}(0) \\
		&\qquad - \sum_{m \geq 0} \Bigg( \sum_{a=1}^{i} W_{i+1,a}^{(1)}(-m) W_{a,i}^{(1)}(m) + \sum_{a=i+1}^{n} W_{i+1,a}^{(1)}(-m-1) W_{a,i}^{(1)}(m+1) \Bigg)\Bigg)
	\end{split}
\end{equation*}
for $i=1,\ldots,n-1$.
\end{prop}

\begin{proof}[Proof of Proposition~\ref{prop:image}]
We show the assertions only for $i = 0$.
Note that we have an equality $\square^{l-1}(X) = \sum_{s=1}^{l} X^{[s]}$ in $U(\hat{\gl}_{n}^{\alpha})_{{\mathrm{comp}}}^{\otimes l}$.

We show the assertion for $X_{0,1}^{+}$.
We have
\begin{equation*}
	\begin{split}
		&\ev^{\otimes l} \circ \Delta^{l-1}(X_{0,1}^{+}) = \square^{l-1}(\ev(X_{0,1}^{+})) \\
		&\ +\hbar \sum_{1 \leq s_1 < s_2 \leq l} \ev^{\otimes l} \Bigg( E_{n,1}^{[s_1]}(1) (c^{[s_2]}) + \sum_{m \geq 0} \sum_{a=1}^{n} \Big( E_{a,1}^{[s_1]}(m+1) E_{n,a}^{[s_2]}(-m) - E_{n,a}^{[s_1]}(m+1) E_{a,1}^{[s_2]}(-m) \Big)\Bigg)\\
		&= (-\hbar) \Bigg( - \alpha \sum_{s=1}^{l} E_{n,1}^{[s]}(1) - \sum_{m \geq 0} \sum_{s=1}^{l} \sum_{a=1}^{n} E_{n,a}^{[s]}(-m) E_{a,1}^{[s]}(m+1) - \alpha \sum_{s=1}^{l} (l-s) E_{n,1}^{[s]}(1) \\
		&\qquad\qquad\qquad\qquad - \sum_{m \geq 0} \sum_{1 \leq s_1 < s_2 \leq l} \sum_{a=1}^{n} \Big( E_{a,1}^{[s_1]}(m+1) E_{n,a}^{[s_2]}(-m) - E_{n,a}^{[s_1]}(m+1) E_{a,1}^{[s_2]}(-m) \Big)\Bigg).
	\end{split}
\end{equation*}
This is equal to
\begin{equation*}
	\begin{split}
		&(-\hbar) \Bigg( - \alpha \sum_{s=1}^{l} E_{n,1}^{[s]}(1) - \alpha \sum_{s=1}^{l} (l-s) E_{n,1}^{[s]}(1) + \sum_{m \in \bbZ} \sum_{1 \leq s_1 < s_2 \leq l} \sum_{a=1}^{n} E_{n,a}^{[s_1]}(m+1) E_{a,1}^{[s_2]}(-m) \\
		&\qquad\qquad - \sum_{m \geq 0} \sum_{a=1}^{n} \left( \sum_{s_1=1}^{l} E_{n,a}^{[s_1]}(-m) \right) \left( \sum_{s_2=1}^{l} E_{a,1}^{[s_2]}(m+1) \right) \Bigg).
	\end{split}
\end{equation*}
Then we obtain
\begin{equation}
	\begin{split}
		&\eta^{(l)} \circ \ev^{\otimes l} \circ \Delta^{l-1}(X_{0,1}^{+}) = \ev^{\otimes l} \circ \Delta^{l-1}(X_{0,1}^{+}) \\
		&\quad + (-\hbar) \Bigg( \sum_{1 \leq s_1 < s_2 \leq l} \Big( (l-s_2)\alpha E_{n,1}^{[s_1]}(1) + (l-s_1)\alpha E_{n,1}^{[s_2]}(1) \Big) - \sum_{s_1=1}^{l} (l-s_1)\alpha \sum_{s_2=1}^{l} E_{n,1}^{[s_2]}(1) \Bigg)\\
		&= (-\hbar) \Bigg( - \alpha \sum_{s=1}^{l} E_{n,1}^{[s]}(1) - \alpha \sum_{s=1}^{l} (l-s) E_{n,1}^{[s]}(1) + \sum_{m \in \bbZ} \sum_{a=1}^{n} \sum_{1 \leq s_1 < s_2 \leq l} E_{n,a}^{[s_1]}(m+1) E_{a,1}^{[s_2]}(-m) \\
		&\qquad\qquad - \sum_{m \geq 0} \sum_{a=1}^{n} \left( \sum_{s_1=1}^{l} E_{n,a}^{[s_1]}(-m) \right) \left( \sum_{s_2=1}^{l} E_{a,1}^{[s_2]}(m+1) \right) - \alpha \sum_{s=1}^{l} (l-s) E_{n,1}^{[s]}(1) \Bigg) \label{eq:Phi}.
	\end{split}
\end{equation}
Since we have
\begin{equation*}
	\begin{split}
		W_{i,j}^{(1)}(m) &= \sum_{s=1}^{l} E_{i,j}^{[s]}(m),\\
		W_{n,1}^{(2)}(1) &= \sum_{m \in \bbZ} \sum_{1 \leq s_1 < s_2 \leq l} \sum_{a=1}^{n} E_{n,a}^{[s_1]}(m+1) E_{a,1}^{[s_2]}(-m) - 2\alpha \sum_{s=1}^{l} (l-s) E_{n,1}^{[s]}(1)
	\end{split}
\end{equation*}
by Example~\ref{exam:generators}, we see that (\ref{eq:Phi}) is equal to 
\[
	(-\hbar) \Bigg( W_{n,1}^{(2)}(1) - \alpha W_{n,1}^{(1)}(1) - \sum_{m \geq 0} \sum_{a=1}^{n} W_{n,a}^{(1)}(-m) W_{a,1}^{(1)}(m+1)\Bigg).
\]

We show the assertion for $X_{0,1}^{-}$.
We have
\begin{equation*}
	\begin{split}
		&\ev^{\otimes l} \circ \Delta^{l-1}(X_{0,1}^{-}) = \square^{l-1}(\ev(X_{0,1}^{-})) \\
		&\ + \hbar \sum_{1 \leq s_1 < s_2 \leq l} \ev^{\otimes l} \Bigg( (c^{[s_1]}) E_{1,n}^{[s_2]}(-1) + \sum_{m \geq 0} \sum_{a=1}^{n} \Big( E_{a,n}^{[s_1]}(m) E_{1,a}^{[s_2]}(-m-1) - E_{1,a}^{[s_1]}(m) E_{a,n}^{[s_2]}(-m-1) \Big)\Bigg)\\
		&= (-\hbar) \Bigg( - \alpha \sum_{s=1}^{l} E_{1,n}^{[s]}(-1) - \sum_{m \geq 0} \sum_{s=1}^{l} \sum_{a=1}^{n} E_{1,a}^{[s]}(-m-1) E_{a,n}^{[s]}(m) - \alpha \sum_{s=1}^{l} (s-1) E_{1,n}^{[s]}(-1) \\
		&\qquad\qquad\qquad\qquad - \sum_{m \geq 0} \sum_{1 \leq s_1 < s_2 \leq l} \sum_{a=1}^{n} \Big( E_{a,n}^{[s_1]}(m) E_{1,a}^{[s_2]}(-m-1) - E_{1,a}^{[s_1]}(m) E_{a,n}^{[s_2]}(-m-1) \Big)\Bigg).
	\end{split}
\end{equation*}
This is equal to
\begin{equation*}
	\begin{split}
		&(-\hbar) \Bigg( - l \alpha \sum_{s=1}^{l} E_{1,n}^{[s]}(-1) + \alpha \sum_{s=1}^{l} (l-s) E_{1,n}^{[s]}(-1) + \sum_{m \in \bbZ} \sum_{1 \leq s_1 < s_2 \leq l} \sum_{a=1}^{n} E_{1,a}^{[s_1]}(m) E_{a,n}^{[s_2]}(-m-1) \\
		&\qquad\qquad - \sum_{m \geq 0} \sum_{a=1}^{n} \left( \sum_{s_1=1}^{l} E_{1,a}^{[s_1]}(-m-1) \right) \left( \sum_{s_2=1}^{l} E_{a,n}^{[s_2]}(m) \right) \Bigg).
	\end{split}
\end{equation*}
Then we obtain
\begin{equation}
	\begin{split}
		&\eta^{(l)} \circ \ev^{\otimes l} \circ \Delta^{l-1}(X_{0,1}^{-}) = \ev^{\otimes l} \circ \Delta^{l-1}(X_{0,1}^{-}) \\
		&\quad + (-\hbar) \Bigg( \sum_{1 \leq s_1 < s_2 \leq l} \Big( (l-s_2)\alpha E_{1,n}^{[s_1]}(-1) + (l-s_1)\alpha E_{1,n}^{[s_2]}(-1) \Big) - \sum_{s_2=1}^{l} (l-s_2)\alpha \sum_{s_1=1}^{l} E_{1,n}^{[s_1]}(-1) \Bigg)\\
		&= (-\hbar) \Bigg( - l \alpha \sum_{s=1}^{l} E_{1,n}^{[s]}(-1) + \alpha \sum_{s=1}^{l} (l-s) E_{1,n}^{[s]}(-1) + \sum_{m \in \bbZ} \sum_{a=1}^{n} \sum_{1 \leq s_1 < s_2 \leq l} E_{1,a}^{[s_1]}(m) E_{a,n}^{[s_2]}(-m-1) \\
		&\qquad\qquad\quad - \sum_{m \geq 0} \sum_{a=1}^{n} \left( \sum_{s_1=1}^{l} E_{1,a}^{[s_1]}(-m-1) \right) \left( \sum_{s_2=1}^{l} E_{a,n}^{[s_2]}(m) \right) - \alpha \sum_{s=1}^{l} (l-s) E_{1,n}^{[s]}(-1) \Bigg) \label{eq:Phi2}.
	\end{split}
\end{equation}
Since we have
\begin{equation*}
		W_{1,n}^{(2)}(-1) = \sum_{m \in \bbZ} \sum_{1 \leq s_1 < s_2 \leq l} \sum_{a=1}^{n} E_{1,a}^{[s_1]}(m) E_{a,n}^{[s_2]}(-m-1)
\end{equation*}
by Example~\ref{exam:generators}, we see that (\ref{eq:Phi2}) is equal to 
\[
	(-\hbar) \Bigg( W_{1,n}^{(2)}(-1) - l \alpha W_{1,n}^{(1)}(-1) - \sum_{m \geq 0} \sum_{a=1}^{n} W_{1,a}^{(1)}(-m-1) W_{a,n}^{(1)}(m) \Bigg).
\]
\end{proof}
\noindent \textbf{Continuation of the proof of Theorem~\ref{thm:main}}:
The map $\Phi_{l}$ coincides with $\Phi$ constructed by the second named author in \cite{Affine_super_Yangians_and_rectangular_W-superalgebras} up to the twist described in Remark~\ref{rem:difference} (i).
Therefore the surjectivity of $\Phi_{l}$ follows from \cite{Affine_super_Yangians_and_rectangular_W-superalgebras}.
We give some part of the discussion in Appendix~\ref{section:surjectivity} for completeness.
\end{proof}

\begin{cor}
Let $M= \bigoplus_{\gamma \in \bbC} M_{\gamma}$ be a graded $\mathfrak{U}(\rect)$-module and for each $\gamma$ we assume that $M_{\gamma + d} = 0$ holds for sufficiently large $d \in \bbZ$.
Then $M$ is regarded as a $\affY$-module via $\Phi_{l}$.

Moreover, each $M_{\gamma}$ is regarded as a module of the Yangian $Y(\mathfrak{sl}_{n})$ of finite type by restriction.
\end{cor}
\begin{proof}
The former assertion is immediate.
For the latter one, we see that the subalgebra of $\affY$ generated by $X_{i,r}^{\pm}$, $H_{i,r}$ ($i=1,\ldots,n-1$ and $r \in \bbZ_{\geq 0}$) is isomorphic to the Yangian $Y(\mathfrak{sl}_{n})$ of finite type and lies in the degree zero component.
Thus $\Phi_{l}(Y(\mathfrak{sl}_{n}))$ lies in the degree zero component of $\mathfrak{U}(\rect)$.
This proves the assertion.
\end{proof} 

\section{Coproduct and parabolic induction}\label{section:Coproduct_and_parabolic_induction}

Assume $l=l_1+l_2$.
Then $\mathcal{W}^{\kappa_i}(\gl_{N_i},f_i)$ for $i=1,2$ were defined in Section~\ref{section:Parabolic_induction}. 
Let
\[
	\Delta_{l_1,l_2} \colon \mathfrak{U}(\rect) \to \mathfrak{U}(\mathcal{W}^{\kappa_1}(\gl_{N_1},f_1)) \otimes_{\mathrm{comp}} \mathfrak{U}(\mathcal{W}^{\kappa_2}(\gl_{N_2},f_2))
\]
be the injective algebra homomorphism induced from the identity map on $U(\hat{\gl}_{n}^{\alpha})_{\mathrm{comp}}^{\otimes l}$.
Here $\otimes_{\mathrm{comp}}$ denotes the degree-wise completed tensor product defined in the same way as one in Section~\ref{section:Affine_Lie_algebras}. 
We give a compatibility of the coproduct $\Delta$ for $\affY$ and $\Delta_{l_1,l_2}$.
Recall the algebra automorphism $\eta_{\beta}$ ($\beta \in \bbC$) of $U(\hat{\gl}_{n}^{\alpha})$ defined in Section~\ref{section:Affine_Yangians_and_rectangular_W-algebras}.

\begin{lem}\label{lem:eta}
For any $\beta$, the algebra automorphism $\eta_{\beta}^{\otimes l}$ preserves $\mathfrak{U}(\rect)$ inside $U(\hat{\gl}_{n}^{\alpha})_{\mathrm{comp}}^{\otimes l}$.
\end{lem}

\begin{proof}
First let us consider the case $\alpha \neq 0$.
Proposition~\ref{prop:generators} asserts that $\mathfrak{U}(\rect)$ is generated by $W_{i,j}^{(r)}(m)$ for $r=1,2$ and all $i,j,m$.
By a direct computation, we have
\begin{align*}
	\eta_{\beta}^{\otimes l}(W_{i,j}^{(1)}(m)) &= W_{i,j}^{(1)}(m) + \delta_{m,0}\delta_{i,j} \beta l, \\
	\eta_{\beta}^{\otimes l}(W_{i,j}^{(2)}(m)) &= W_{i,j}^{(2)}(m) + \beta (l-1) W_{i,j}^{(1)}(m) + \delta_{m,0}\delta_{i,j} \beta(\beta-\alpha) \dfrac{l(l-1)}{2},
\end{align*}
which proves the assertion.

Next consider the case $\alpha = 0$.
In this case, we have
\[
	W_{i,j}^{(r)}(m) = \sum_{\substack{m_1,\ldots,m_r \in \bbZ\\ m_1+\cdots+m_r=m}}\ \sum_{1 \leq s_1 < \cdots < s_r \leq l}\ \sum_{a_1,\ldots,a_{r-1}=1}^{n} E_{i,a_1}^{[s_1]}(m_1) E_{a_1,a_2}^{[s_2]}(m_2) \cdots E_{a_{r-1},j}^{[s_r]}(m_r)
\]
for $r=1,\ldots,l$, and $\mathfrak{U}(\rect)$ is generated by these elements.
By a direct computation, we have
\[
	\eta_{\beta}^{\otimes l}(W_{i,j}^{(r)}(m)) = \sum_{s=0}^{r} \beta^{s} 
	\begin{pmatrix}
		l-r+s\\
		s
	\end{pmatrix} W_{i,j}^{(r-s)}(m),
\]
where $W_{i,j}^{(0)}(m)$ is regarded as $\delta_{m,0}\delta_{i,j}$.
This proves the assertion.
\end{proof}

Let us define $\widetilde{\Delta}_{l_1,l_2}$, a slight modification of $\Delta_{l_1,l_2}$, by
\[
	\widetilde{\Delta}_{l_1,l_2} = (\eta_{-l_2 \alpha}^{\otimes l_1} \otimes \id^{\otimes l_2}) \circ \Delta_{l_1,l_2}. 
\]
This gives an injective algebra homomorphism from $\mathfrak{U}(\rect)$ to $\mathfrak{U}(\mathcal{W}^{\kappa_1}(\gl_{N_1},f_1)) \otimes_{\mathrm{comp}} \mathfrak{U}(\mathcal{W}^{\kappa_2}(\gl_{N_2},f_2))$ since we have $\eta_{-l_2 \alpha}^{\otimes l_1} \mathfrak{U}(\mathcal{W}^{\kappa_1}(\gl_{N_1},f_1)) = \mathfrak{U}(\mathcal{W}^{\kappa_1}(\gl_{N_1},f_1))$ by applying Lemma~\ref{lem:eta}.
\begin{cor}\label{cor:parabolic_induction}
We have
\[
	\widetilde{\Delta}_{l_1,l_2} \circ \Phi_{l} = \left( \Phi_{l_1} \otimes \Phi_{l_2} \right) \circ \Delta.
\]
\end{cor}
\begin{proof}
The assertion is proved by
\begin{equation*}
	\begin{split}
		\widetilde{\Delta}_{l_1,l_2} \circ \Phi_{l} &= (\eta_{-l_2 \alpha}^{\otimes l_1} \otimes \id^{\otimes l_2}) \circ \eta^{(l)} \circ \ev^{\otimes l} \circ \Delta^{l-1}\\
		&= (\eta^{(l_1)} \otimes \eta^{(l_2)}) \circ \ev^{\otimes l} \circ \Delta^{l-1} = \left( \Phi_{l_1} \otimes \Phi_{l_2} \right) \circ \Delta.
	\end{split}
\end{equation*}
\end{proof}

We can reformulate our main theorem and the above compatibility as follows.
Put
\[
	{}^\eta \mathfrak{U}(\rect) = ( \eta^{(l)} )^{-1} \mathfrak{U}(\rect), \ \ {}^\eta \mathfrak{U}(\mathcal{W}^{\kappa_i}(\gl_{N_i},f_i)) = ( \eta^{(l_i)} )^{-1} \mathfrak{U}(\mathcal{W}^{\kappa_i}(\gl_{N_i},f_i)).
\]
Since we have
\begin{equation*}
	\begin{split}
		( \eta^{(l)} )^{-1} \mathfrak{U}(\rect) &= (\eta^{(l_1)} \otimes \eta^{(l_2)})^{-1} \circ (\eta_{-l_2 \alpha}^{\otimes l_1} \otimes \id^{\otimes l_2})\, \mathfrak{U}(\rect) \\
		&\subset (\eta^{(l_1)} \otimes \eta^{(l_2)})^{-1} \circ (\eta_{-l_2 \alpha}^{\otimes l_1} \otimes \id^{\otimes l_2})\, \mathfrak{U}(\mathcal{W}^{\kappa_1}(\gl_{N_1},f_1)) \otimes_{\mathrm{comp}} \mathfrak{U}(\mathcal{W}^{\kappa_2}(\gl_{N_2},f_2))\\
		&= (\eta^{(l_1)} \otimes \eta^{(l_2)})^{-1} \mathfrak{U}(\mathcal{W}^{\kappa_1}(\gl_{N_1},f_1)) \otimes_{\mathrm{comp}} \mathfrak{U}(\mathcal{W}^{\kappa_2}(\gl_{N_2},f_2)),
	\end{split}
\end{equation*}
we see that ${}^\eta \mathfrak{U}(\rect)$ is contained in ${}^\eta \mathfrak{U}(\mathcal{W}^{\kappa_1}(\gl_{N_1},f_1)) \otimes_{\mathrm{comp}} {}^\eta \mathfrak{U}(\mathcal{W}^{\kappa_2}(\gl_{N_2},f_2))$.
We denote by ${}^\eta \Delta_{l_1,l_2}$ this inclusion.
Put
\[
	{}^{\eta} \Phi_{l} = ( \eta^{(l)} )^{-1} \circ \Phi_{l} = \ev^{\otimes l} \circ \Delta^{l-1}.
\]
Now the following is obvious.
\begin{cor}\label{cor:parabolic_induction2}
The map ${}^{\eta} \Phi_{l} = \ev^{\otimes l} \circ \Delta^{l-1}$ gives an algebra homomorphism
\[
	{}^{\eta} \Phi_{l} \colon \affY_{\mathrm{comp}} \to {}^\eta \mathfrak{U}(\rect)
\]
satisfying
\[
	{}^{\eta} \Delta_{l_1,l_2} \circ {}^{\eta} \Phi_{l} = \left( {}^{\eta} \Phi_{l_1} \otimes {}^{\eta} \Phi_{l_2} \right) \circ \Delta.
\]
Moreover, it is surjective if $\alpha \neq 0$.
\end{cor}

\section{Commuting elements}\label{section:Commuting_elements}

By our main result, the elements $\Phi_{l}(H_{i,r})$ for $i= 0, 1, \ldots, n-1$ and $r \in \bbZ_{\geq 0}$ generate a commutative subalgebra of $\mathfrak{U}(\rect)$.
We expect that the subalgebra will play an important role for applications to the study of integrable systems.
Let us mention a few observation for commuting elements of low degree here.

Put
\begin{equation*}
	D_{i} = W_{i,i}^{(2)}(0) - \sum_{m \geq 0} \left( \sum_{a=1}^{i-1} W_{i,a}^{(1)}(-m) W_{a,i}^{(1)}(m) + \sum_{a=i}^{n} W_{i,a}^{(1)}(-m-1) W_{a,i}^{(1)}(m+1) \right)
\end{equation*}
for $i=1,\ldots,n$.
This element is an analog of $T_{i,i;i-1}^{(2)}$ in \cite[Example~9.1]{MR2199632}
(See also \cite[Section~3.3 (3.10) and Section~3.4]{MR2456464}).
Note that $D_{i}$ is also equal to
\begin{equation*}
	W_{i,i}^{(2)}(0) - \sum_{m \geq 0} \left( \sum_{a=1}^{i} W_{i,a}^{(1)}(-m) W_{a,i}^{(1)}(m) + \sum_{a=i+1}^{n} W_{i,a}^{(1)}(-m-1) W_{a,i}^{(1)}(m+1) \right) + W_{i,i}^{(1)}(0)^2.
\end{equation*}
We can write down $\Phi_{l}(H_{i,1})$ in terms of $D_{i}$'s as:
\begin{equation}
	\begin{split}
		&\Phi_{l}(H_{i,1}) = (-\hbar) \times \\
		&= \begin{cases}
			\Bigg( D_{n} - D_{1} - l\alpha \Big( W_{n,n}^{(1)}(0) - W_{1,1}^{(1)}(0) + l\alpha \Big) - W_{n,n}^{(1)}(0)^2 + W_{n,n}^{(1)}(0) \Big( W_{1,1}^{(1)}(0) - \alpha \Big) \Bigg) & (i = 0),\\
			\Bigg( D_{i} - D_{i+1} + \dfrac{i}{2} \Big( W_{i,i}^{(1)}(0) - W_{i+1,i+1}^{(1)}(0) \Big) - W_{i,i}^{(1)}(0)^2 + W_{i,i}^{(1)}(0) W_{i+1,i+1}^{(1)}(0) \Bigg) & (i \neq 0).
		\end{cases} \label{eq:image_of_H}
	\end{split}
\end{equation}
We note that (\ref{eq:image_of_H}) coincides with the formula for the evaluation map in Theorem~\ref{thm:evaluation} when $l=1$, if we put $W_{i,i}^{(2)}(0) = 0$.

By a lengthy computation, we see that $[D_{i}, D_{j}] = 0$ holds.

\begin{prop}\label{prop:commuting_elements}
We have $[D_{i}, D_{j}] = 0$.
\end{prop}
\begin{proof}
A proof is given in Appendix~\ref{section:commutativity}.
\end{proof}

Moreover we easily see that $[D_{i}, W_{j,j}^{(1)}(0)] = 0$ holds.
Thus these computations imply that the right-hand sides of (\ref{eq:image_of_H}) for any $i,j$ mutually commute without using our main result.
This approach was used for a direct construction of $\Phi_{l}$ by the second named author in \cite{Affine_super_Yangians_and_rectangular_W-superalgebras}. 

Let us say a few words on Proposition~\ref{prop:commuting_elements}.  
If we put
\[
	A_{i} = \sum_{m \geq 0} \sum_{a=1}^{i-1} W_{i,a}^{(1)}(-m) W_{a,i}^{(1)}(m),\quad B_{i} = \sum_{m \geq 0} \sum_{a=i}^{n} W_{i,a}^{(1)}(-m-1) W_{a,i}^{(1)}(m+1),
\]
then computations for
\[
	[W_{i,i}^{(2)}(0), W_{j,j}^{(2)}(0)], \quad [W_{i,i}^{(2)}(0), A_{j}+B_{j}], \quad [A_{i}+B_{i},A_{j}+B_{j}]
\]
yield the vanishing of $[D_{i}, D_{j}]$.
We have
\begin{equation}
	[A_{i}+B_{i},A_{j}+B_{j}] = (l-1) \sum_{m \geq 1} m \Big( W_{i,i}^{(1)}(-m)W_{j,j}^{(1)}(m) - W_{j,j}^{(1)}(-m)W_{i,i}^{(1)}(m) \Big). \label{eq:AB}
\end{equation}
We write down $[W_{i,i}^{(2)}(0), W_{j,j}^{(2)}(0)]$ and $[W_{i,i}^{(2)}(0), A_{j}+B_{j}]$, and give a proof of Proposition~\ref{prop:commuting_elements} in Appendix~\ref{section:commutativity} for readers' convenience.

In the case $l=1$, the map $\Phi_{1}$ coincides with the evaluation map $\ev$,
and we regard $W_{i,j}^{(2)}(m)$ as $0$.
Then (\ref{eq:AB}) for $l=1$ implies that $[\ev(H_{i,1}), \ev(H_{j,1})]=0$ holds.
This is the main part of computations done by the first named author in \cite{MR4207398} to check that $\ev$ respects the defining relations of $\affY$. 
See Appendix~\ref{section:proof_of_evaluation_map} for a derivation of (\ref{eq:AB}).
As mentioned in Remark~\ref{rem:error}, the published version of \cite{MR4207398} has an error.
Although the correction to \cite{MR4207398} is available, we include some explanation in this paper since we use conventions slightly different from \cite{MR4207398}.

\section{Super case}\label{section:Super_case}

Let $m,n$ be nonzero positive integers with $m,n \geq 2$ and $m \neq n$.
Set $M=lm$ and $N=ln$.
For a fixed complex number $k$, we put $\alpha = k+(M-N)-(m-n) = k+(l-1)(m-n)$.

We set $\pari{i}$ to be $0$ if $1 \leq i \leq m$ and to be $1$ if $m+1 \leq i \leq m+n$.
Let $\mathfrak{gl}_{m|n}$ denote the general linear Lie superalgebra with the standard parity.
That is, the parity of the element $E_{i,j}$ ($1 \leq i,j \leq m+n$) is given by $\pari{i}+\pari{j}$.
We define an affinization $\hat{\mathfrak{gl}}_{m|n}^{\alpha}$ of $\mathfrak{gl}_{m|n}$ with generators $E_{i,j}(d)$ ($1 \leq i,j \leq m+n$ and $d \in \bbZ$) subject to the following relation:
\begin{multline*}
	[E_{i,j}(d), E_{p,q}(d')] \\
	= \delta_{pj}E_{i,q}(d+d')-\delta_{iq}(-1)^{(\pari{i}+\pari{j})(\pari{p}+\pari{q})}E_{p,j}(d+d') + d \delta_{d+d',0}\Big( \delta_{iq}\delta_{jp}(-1)^{\pari{i}}\alpha + \delta_{ij}\delta_{pq}(-1)^{\pari{i}+\pari{p}} \Big).
\end{multline*}
An algebra automorphism $\eta_{\beta}$ ($\beta \in \bbC$) of $U(\hat{\gl}_{m|n}^{\alpha})$ is defined by
\[
	E_{i,j}(d) \mapsto E_{i,j}(d) + \delta_{d,0}\delta_{i,j}(-1)^{\pari{i}}\beta.
\] 

We define a $\bbZ/2\bbZ$-grading of $\frak{g} = \mathfrak{gl}_{M|N} \cong \mathfrak{gl}_{l} \otimes \mathfrak{gl}_{m|n}$ by
\begin{equation}
	\frak{g}_{\pari{j}} = \mathfrak{gl}_{l} \otimes \left( \mathfrak{gl}_{m|n} \right)_{\pari{j}} \label{eq:parity} 
\end{equation}
for $j=0,1$.
Then the parity of $E_{s,t} \otimes E_{i,j}$ is $\pari{i}+\pari{j}$.
This is not the standard parity if $l \geq 2$.
One can define the rectangular $W$-superalgebra $\rect$ attached with
\begin{equation*}
	\frak{g} = \mathfrak{gl}_{M|N} \cong \mathfrak{gl}_{l} \otimes \mathfrak{gl}_{m|n}, \ f = \sum_{s=1}^{l-1} \sum_{i=1}^{m+n} E_{s(m+n)+i,(s-1)(m+n)+i} = \left( \sum_{s=1}^{l-1} E_{s+1,s} \right) \otimes I_{m+n},
\end{equation*}
and a supersymmetric bilinear form $\kappa$ on $\frak{g} = \mathfrak{sl}_{M|N} \oplus \frz_{M+N}$ defined by
\begin{equation*}
	\kappa(X,Y) = \begin{cases}
		k \str_{\frak{g}}(XY) & \text{if } X,Y \in \fraksl_{M \mid N},\\
		(k+M-N)(M-N) & \text{if } X=Y = I_{M+N},\\
		0 & \text{if } X \in \fraksl_{M \mid N} \text{ and } Y \in \frz_{M+N}.
	\end{cases}
\end{equation*}
Note that the supertrace is taken with respect to the parity given by (\ref{eq:parity}).
The explicit value of $\kappa$ is
\[
	\kappa(E_{s_1,t_1} \otimes E_{i,j}, E_{s_2,t_2} \otimes E_{p,q}) = \delta_{s_1,t_2}\delta_{t_1,s_2} \delta_{iq}\delta_{jp}(-1)^{\pari{i}}k + \delta_{s_1,t_1}\delta_{s_2,t_2} \delta_{ij}\delta_{pq}(-1)^{\pari{i}+\pari{p}}.
\]
See \cite{MR4088277} and \cite{Affine_super_Yangians_and_rectangular_W-superalgebras} for details.
Its current algebra $\mathfrak{U}(\rect)$ is identified with a subalgebra of $U(\hat{\mathfrak{gl}}_{m|n}^{\alpha})_{\mathrm{comp}}^{\otimes l}$ via the Miura map. 
The injectivity of the Miura map for $W$-superalgebras has been established by Genra \cite{MR3663604} for generic levels and by Nakatsuka \cite{2005.10472} for general levels.

By \cite{Affine_super_Yangians_and_rectangular_W-superalgebras}, one can define generators $W_{i,j}^{(r)}(d)$ ($1 \leq r \leq l, \ 1 \leq i,j \leq m+n, \ d \in \bbZ$) of $\mathfrak{U}(\rect)$ as follows.
Let $A^{[s]} = \left( (-1)^{\pari{i}}E_{i,j}^{[s]}[-1]_{(-1)} \right)_{1 \leq i,j \leq m+n}$ be the $(m+n) \times (m+n)$ matrix whose $(i,j)$-entry is $(-1)^{\pari{i}}E_{i,j}^{[s]}[-1]_{(-1)}$.
We define a matrix $W^{(r)}$ by
\begin{equation*}
	(\tau + A^{[1]}) (\tau + A^{[2]}) \cdots (\tau + A^{[l]}) = \tau^{l} + \sum_{r=1}^{l} \tau^{l-r} W^{(r)}.
\end{equation*}
Then generators $W_{i,j}^{(r)}$ of the rectangular $W$-superalgebra $\rect$ are defined by
\[
	W^{(r)} = \left( (-1)^{\pari{i}}W_{i,j}^{(r)} \right)_{1 \leq i,j \leq m+n}.
\]
We set
\[
	W_{i,j}^{(r)}(d) = W_{i,j}^{(r)} t^{d+r-1} \in \mathfrak{U}(\rect).
\]
Moreover, if $\alpha \neq 0$ then the elements of the generators for $r=1,2$ are enough to generate $\mathfrak{U}(\rect)$. 
\begin{exam}\label{exam:generators_super}
We have
\begin{align*}
	W_{i,j}^{(1)}(d) &= \sum_{s=1}^{l} E_{i,j}^{[s]}(d),\\
	W_{i,j}^{(2)}(d) &= \sum_{r \in \bbZ} \sum_{1 \leq s_1 < s_2 \leq l} \sum_{a=1}^{m+n} (-1)^{\pari{a}} E_{i,a}^{[s_1]}(d+r) E_{a,j}^{[s_2]}(-r) - (d+1)\alpha \sum_{s=1}^{l} (l-s) E_{i,j}^{[s]}(d).
\end{align*}
\end{exam}

We similarly introduce parabolic inductions $\Delta_{l_1,l_2}$ and $\tilde{\Delta}_{l_1,l_2}$ for a decomposition $l=l_1+l_2$.
They are injective homomorphisms from $\mathfrak{U}(\rect)$ to $\mathfrak{U}(\mathcal{W}^{\kappa_1}(\mathfrak{gl}_{M_1 \mid N_1}, f_1)) \otimes_{\mathrm{comp}} \mathfrak{U}(\mathcal{W}^{\kappa_2}(\mathfrak{gl}_{M_2 \mid N_2}, f_2))$, where $M_i = l_i m$, $N_i = l_i n$ ($i=1,2$).
An algebra automorphism $\eta^{(l)}$ of $U(\hat{\gl}_{m|n}^{\alpha})_{{\mathrm{comp}}}^{\otimes l}$ is defined by the same way as in the non-super case.

The affine super Yangian $Y(\hat{\frak{sl}}_{m|n})$ was defined by the second named author in \cite{Affine_super_Yangian}.
See also \cite[C.2]{2011.01203}.
It has two parameters $\hbar, \ve$ as the affine Yangian $Y(\hat{\frak{sl}}_{n})$.
Its generators are denoted by $X_{i,r}^{+}$, $X_{i,r}^{-}$, $H_{i,r}$ ($i = 0,1,\ldots,m+n-1$ and $r \in \mathbb{Z}_{\geq 0})$.  
Set
\begin{gather*}
	X_{0}^{+} = E_{m+n,1}(1), \quad X_{0}^{-} = -E_{1,m+n}(-1), \quad H_{0} = (-E_{m+n,m+n}) - E_{1,1} + c,\\
	X_{i}^{+} = E_{i,i+1}, \quad X_{i}^{-} = (-1)^{\pari{i}} E_{i+1,i}, \quad H_{i} = (-1)^{\pari{i}}E_{i,i}-(-1)^{\pari{i+1}}E_{i+1,i+1} \ \ (i \neq 0)
\end{gather*}
in the affine Lie superalgebra $\hat{\mathfrak{sl}}_{m|n}$ and use the same symbols for their images in $Y(\hat{\frak{sl}}_{m|n})$.

We prepare notations in order to simplify formulas below.
For $i=1,\ldots,m+n-1$ and $a=1,\ldots,m+n$, we set
\[
	\cpari{i}{a} = \pari{a} + (\pari{i}+\pari{a})(\pari{i+1}+\pari{a}) = \begin{cases}
		0 & \text{ if } 1 \leq i \leq m-1, \\
		\pari{a} & \text{ if } i=m, \\
		1 & \text{ if } m+1 \leq i \leq m+n-1.
	\end{cases}
\]
Fix a sequence $(c_{i})_{1 \leq i \leq m+n-1}$ by $c_{i}=-i/2$ for $1 \leq i \leq m$ and $c_{m+i}=(-m+i)/2$ for $1 \leq i \leq n-1$.
We give formulas for the coproduct $\Delta \colon Y(\hat{\mathfrak{sl}}_{m|n}) \to Y(\hat{\mathfrak{sl}}_{m|n})_{\mathrm{comp}}^{\otimes 2}$ and the evaluation map $\ev \colon Y(\hat{\mathfrak{sl}}_{m | n}) \to U(\hat{\gl}_{m|n}^{\alpha})_{\mathrm{comp}}$ from \cite{Affine_super_Yangian}.
The latter map exists under the assumption $\alpha + m - n = -\ve/\hbar$.
We omit to write down their explicit forms for $H_{i,1}$ here;
\begin{align*}
	&\Delta(X_{0,1}^{+}) = \square(X_{0,1}^{+}) \\
	&\ + \hbar \Bigg( X_{0}^{+} \otimes c + \sum_{d \geq 0} \sum_{a=1}^{m+n} (-1)^{\pari{a}} \Big( E_{a,1}(d+1) \otimes E_{m+n,a}(-d) - E_{m+n,a}(d+1) \otimes E_{a,1}(-d) \Big)\Bigg),\\
	&\Delta(X_{0,1}^{-}) = \square(X_{0,1}^{-}) \\
	&\ + \hbar \Bigg( c \otimes X_{0}^{-} - \sum_{d \geq 0} \sum_{a=1}^{m+n} (-1)^{\pari{a}} \Big( E_{a,m+n}(d) \otimes E_{1,a}(-d-1) - E_{1,a}(d) \otimes E_{a,m+n}(-d-1) \Big)\Bigg)
\end{align*}
and
\begin{align*}
	&\Delta(X_{i,1}^{+}) = \square(X_{i,1}^{+}) + \hbar \sum_{d \geq 0} \\
	&\quad\Bigg( \sum_{a=1}^{i} (-1)^{\cpari{i}{a}} E_{a,i+1}(d) \otimes E_{i,a}(-d) + \sum_{a=i+1}^{m+n} (-1)^{\cpari{i}{a}} E_{a,i+1}(d+1) \otimes E_{i,a}(-d-1) \\
	&\qquad\qquad\quad - \sum_{a=1}^{i} (-1)^{\pari{a}} E_{i,a}(d+1) \otimes E_{a,i+1}(-d-1) - \sum_{a=i+1}^{m+n}  (-1)^{\pari{a}} E_{i,a}(d) \otimes E_{a,i+1}(-d) \Bigg),\\
	&\Delta(X_{i,1}^{-}) = \square(X_{i,1}^{-}) + (-1)^{\pari{i}}\hbar \sum_{d \geq 0} \\
	&\quad\Bigg( \sum_{a=1}^{i} (-1)^{\cpari{i}{a}} E_{a,i}(d) \otimes E_{i+1,a}(-d) + \sum_{a=i+1}^{m+n} (-1)^{\cpari{i}{a}} E_{a,i}(d+1) \otimes E_{i+1,a}(-d-1) \\
	&\qquad\qquad\quad - \sum_{a=1}^{i} (-1)^{\pari{a}} E_{i+1,a}(d+1) \otimes E_{a,i}(-d-1) - \sum_{a=i+1}^{m+n} (-1)^{\pari{a}} E_{i+1,a}(d) \otimes E_{a,i}(-d) \Bigg)
\end{align*}
for $i=1,\ldots,m+n-1$;
\begin{align*}
		\ev(X_{0,1}^{+}) &= \hbar \Bigg( \alpha X_{0}^{+} + \sum_{d \geq 0} \sum_{a=1}^{m+n} (-1)^{\pari{a}} E_{m+n,a}(-d) E_{a,1}(d+1) \Bigg),\\
		\ev(X_{0,1}^{-}) &= \hbar \Bigg( \alpha X_{0}^{-} - \sum_{d \geq 0} \sum_{a=1}^{m+n} (-1)^{\pari{a}} E_{1,a}(-d-1) E_{a,m+n}(d) \Bigg)
\end{align*}
and
\begin{align*}
		&\ev(X_{i,1}^{+}) = \hbar \times\\
		&\ \Bigg( c_{i} X_{i}^{+} + \sum_{d \geq 0} \Bigg( \sum_{a=1}^{i} (-1)^{\pari{a}} E_{i,a}(-d) E_{a,i+1}(d) + \sum_{a=i+1}^{m+n} (-1)^{\pari{a}} E_{i,a}(-d-1) E_{a,i+1}(d+1) \Bigg) \Bigg),\\
		&\ev(X_{i,1}^{-}) = \hbar \times\\
		&\ \Bigg( c_{i} X_{i}^{-} + (-1)^{\pari{i}} \sum_{d \geq 0} \Bigg( \sum_{a=1}^{i} (-1)^{\pari{a}} E_{i+1,a}(-d) E_{a,i}(d) + \sum_{a=i+1}^{m+n} (-1)^{\pari{a}} E_{i+1,a}(-d-1) E_{a,i}(d+1) \Bigg) \Bigg)
\end{align*}
for $i=1,\ldots,m+n-1$.

\begin{dfn}\label{dfn:map_super}
Assume $k+M-N=\alpha+m-n=-\ve/\hbar$.
We set
\[
	\Phi_{l} = \eta^{(l)} \circ \ev^{\otimes l} \circ \Delta^{l-1},
\]
which is an algebra homomorphism from $Y(\hat{\frak{sl}}_{m|n})_{\mathrm{comp}}$ to $U(\hat{\gl}_{m|n}^{\alpha})_{{\mathrm{comp}}}^{\otimes l}$.
\end{dfn}

In the computation of $\Phi_{l}$, we use the formula
\[
	\square(X) \square(Y) = \square(XY) + X \otimes Y + (-1)^{|X||Y|} Y \otimes X,
\]
where $|X|$ and $|Y|$ denote their parities.
We omit details.
\begin{prop}\label{prop:image_super}
We have
\begin{equation*}
	\begin{split}
		&\Phi_{l}(X_{0,1}^{+}) = (-\hbar) \Bigg( W_{m+n,1}^{(2)}(1) - \alpha W_{m+n,1}^{(1)}(1) - \sum_{d \geq 0} \sum_{a=1}^{m+n} (-1)^{\pari{a}} W_{m+n,a}^{(1)}(-d) W_{a,1}^{(1)}(d+1)\Bigg),\\
		&(-1)^{\pari{m+n}} \Phi_{l}(X_{0,1}^{-}) \\
		&= (-\hbar) \Bigg( W_{1,m+n}^{(2)}(-1) - l \alpha W_{1,m+n}^{(1)}(-1) - \sum_{d \geq 0} \sum_{a=1}^{m+n} (-1)^{\pari{a}} W_{1,a}^{(1)}(-d-1) W_{a,m+n}^{(1)}(d) \Bigg)
	\end{split}
\end{equation*}
and
\begin{equation*}
	\begin{split}
		&\Phi_{l}(X_{i,1}^{+}) = (-\hbar) \Bigg( W_{i,i+1}^{(2)}(0) - c_{i} W_{i,i+1}^{(1)}(0) \\
		&\qquad\quad - \sum_{d \geq 0} \Bigg( \sum_{a=1}^{i} (-1)^{\pari{a}} W_{i,a}^{(1)}(-d) W_{a,i+1}^{(1)}(d) + \sum_{a=i+1}^{m+n} (-1)^{\pari{a}} W_{i,a}^{(1)}(-d-1) W_{a,i+1}^{(1)}(d+1) \Bigg)\Bigg),\\
		&(-1)^{\pari{i}} \Phi_{l}(X_{i,1}^{-}) = (-\hbar) \Bigg( W_{i+1,i}^{(2)}(0) - c_{i} W_{i+1,i}^{(1)}(0) \\
		&\qquad\quad - \sum_{d \geq 0} \Bigg( \sum_{a=1}^{i} (-1)^{\pari{a}} W_{i+1,a}^{(1)}(-d) W_{a,i}^{(1)}(d) + \sum_{a=i+1}^{m+n} (-1)^{\pari{a}} W_{i+1,a}^{(1)}(-d-1) W_{a,i}^{(1)}(d+1) \Bigg)\Bigg)
	\end{split}
\end{equation*}
for $i=1,\ldots,m+n-1$.
\end{prop}
Then we obtain a result in the super case by the same argument as in the non-super case.
\begin{thm}\label{thm:introduction_main_super}
Assume $k+M-N=\alpha+m-n=-\ve/\hbar$.
Then the map $\Phi_{l}$ gives an algebra homomorphism
\[
	\Phi_{l} \colon Y(\hat{\frak{sl}}_{m|n})_{\mathrm{comp}} \to \mathfrak{U}(\rect)
\]
which makes the following diagram commutative:
\[
	\xymatrix{
		Y(\hat{\frak{sl}}_{m|n}) \ar[d]_{\Delta} \ar[r]^{\Phi_{l}} \ & \ \mathfrak{U}(\rect) \ar[d]^{\tilde{\Delta}_{l_1,l_2}} \\
		Y(\hat{\frak{sl}}_{m|n})^{\otimes 2}_{\mathrm{comp}} \ar[r]^-{\Phi_{l_1} \otimes \Phi_{l_2}} \ \ & \ \ \mathfrak{U}(\mathcal{W}^{\kappa_1}(\mathfrak{gl}_{M_1 \mid N_1}, f_1)) \otimes_{\mathrm{comp}} \mathfrak{U}(\mathcal{W}^{\kappa_2}(\mathfrak{gl}_{M_2 \mid N_2}, f_2)).
	}	
\]
Moreover, $\Phi_{l}$ is surjective if $\alpha = k+(M-N)-(m-n) \neq 0$.
\end{thm}

\appendix

\section{}\label{section:surjectivity}

Lemma~\ref{lem:OPE2-1} and Proposition~\ref{prop:generators2} below are statements on the rectangular $W$-algebra $\rect$ for $n \geq 1$ and for $n \geq 2$, respectively.

\begin{lem}\label{lem:OPE2-1}
We have
\begin{equation*}
	\begin{split}
		[ W_{i,j}^{(2)}(m), W_{p,q}^{(1)}(m') ] &= \Big( \delta_{pj} W_{i,q}^{(2)}(m+m') - \delta_{iq} W_{p,j}^{(2)}(m+m') \Big) \\
		&\quad - m'(l-1) \Big( \delta_{iq} \alpha W_{p,j}^{(1)}(m+m') + \delta_{pq} W_{i,j}^{(1)}(m+m') \Big) \\
		&\qquad - \dfrac{m'(m'-1)}{2} \delta_{m+m',0} l(l-1)\alpha \Big( \delta_{iq}\delta_{jp}\alpha + \delta_{ij}\delta_{pq} \Big).
	\end{split}
\end{equation*}
In particular, we have
\[
	[W_{i,j}^{(2)}(m),W_{p,q}^{(1)}(0)] = \delta_{pj} W_{i,q}^{(2)}(m) - \delta_{iq} W_{p,j}^{(2)}(m),
\]
\[
	[W_{i,j}^{(2)}(m-1),W_{p,q}^{(1)}(1)] = \Big( \delta_{pj} W_{i,q}^{(2)}(m) - \delta_{iq} W_{p,j}^{(2)}(m) \Big) - (l-1) \Big( \delta_{iq} \alpha W_{p,j}^{(1)}(m) + \delta_{pq} W_{i,j}^{(1)}(m) \Big).
\]
\end{lem}

The following assertion slightly refines Proposition~\ref{prop:generators}.
It can be regarded as an analog of Proposition~\ref{prop:degree_one_generator}.

\begin{prop}\label{prop:generators2}
Assume $l \geq 2,$ $n \geq 2$ and $\alpha \neq 0$.
Then $\mathfrak{U}(\rect)$ is topologically generated by
\[
	\left\{ W_{n,1}^{(1)}(1),\ W_{1,n}^{(1)}(-1),\ W_{i,i+1}^{(1)}(0),\ W_{i+1,i}^{(1)}(0) \mid i=1,\ldots,n-1 \right\} \cup \left\{ W_{i,i+1}^{(2)}(0) \mid i=1,\ldots,n-1 \right\}.
\]
\end{prop}
\begin{proof}
We abbreviate the word ``topologically''.
The argument below is essentially the same as one in \cite{Affine_super_Yangians_and_rectangular_W-superalgebras}.

\noindent $\bullet$ Generate $W_{i,j}^{(1)}(m)$ for $i \neq j$.

They are generated by $\left\{ W_{n,1}^{(1)}(1),\ W_{1,n}^{(1)}(-1),\ W_{i,i+1}^{(1)}(0),\ W_{i+1,i}^{(1)}(0) \mid i=1,\ldots,n-1 \right\}$.

\noindent $\bullet$ Generate $W_{i,i}^{(2)}(0) - W_{j,j}^{(2)}(0)$.

For $i=1,\ldots,n-1$, we have
\begin{equation*}
	\begin{split}
		[W_{i,i+1}^{(2)}(0), W_{i+1,i}^{(1)}(0)] = W_{i,i}^{(2)}(0) - W_{i+1,i+1}^{(2)}(0).
	\end{split}
\end{equation*}

\noindent $\bullet$ Generate $W_{i,j}^{(2)}(m)$ for $i \neq j$.

For $i \neq j$, we have
\begin{equation*}
	\begin{split}
		[W_{i,i}^{(2)}(0) - W_{j,j}^{(2)}(0), W_{i,j}^{(1)}(m)] = 2 W_{i,j}^{(2)}(m) + m(l-1)\alpha W_{i,j}^{(1)}(m).
	\end{split}
\end{equation*}

\noindent $\bullet$ Generate $W_{i,i}^{(2)}(m)-W_{j,j}^{(2)}(m)$ for $i \neq j$.

For $i \neq j$, we have
\begin{equation*}
	\begin{split}
		[W_{i,j}^{(2)}(m), W_{j,i}^{(1)}(0)] = W_{i,i}^{(2)}(m)-W_{j,j}^{(2)}(m).
	\end{split}
\end{equation*}

\noindent $\bullet$ Generate $W_{i,i}^{(1)}(m)$.

For $i \neq j$, we have
\begin{equation*}
	\begin{split}
		[W_{j,i}^{(2)}(m-1), W_{i,j}^{(1)}(1)] = W_{j,j}^{(2)}(m)-W_{i,i}^{(2)}(m) - (l-1)\alpha W_{i,i}^{(1)}(m).
	\end{split}
\end{equation*}

\noindent $\bullet$ Generate $W_{i,i}^{(2)}(m)$.

For $i \neq j$, we have
\[
	[W_{i,i}^{(2)}(m) - W_{j,j}^{(2)}(m), W_{i,i}^{(2)}(m') - W_{j,j}^{(2)}(m')] 
	= (m'-m)\alpha \Big( W_{i,i}^{(2)}(m+m') + W_{j,j}^{(2)}(m+m') \Big) + P,
\]
where $P$ is an element of $\mathfrak{U}(\rect)$ which is generated by $W_{a,b}^{(1)}(m'')$ and $W_{c,d}^{(2)}(m''')$ for various $a,b,c,d,m'',m'''$ with $c \neq d$.
Hence under the assumption $\alpha \neq 0$, we see that all the elements of the form $W_{i,i}^{(2)}(m) + W_{j,j}^{(2)}(m)$ for $i \neq j$ and $m \in \bbZ$ belong to the image of $\Phi_{l}$.
Thus $W_{i,i}^{(2)}(m)$ for any $i$ and $m$ belong to the image of $\Phi_{l}$ since so do $W_{i,i}^{(2)}(m) - W_{j,j}^{(2)}(m)$.
\end{proof}

Assume $\alpha \neq 0$.
Let us prove the surjectivity of $\Phi_{l}$, the latter statement of Theorem~\ref{thm:main}.
By (\ref{eq:image_of_degree_zero}), we see that the image of $\Phi_{l}$ contains $W_{i,j}^{(1)}(m)$ and $W_{i,i}^{(1)}(m)-W_{j,j}^{(1)}(m)$ for $i \neq j$ and $m \in \bbZ$.
The image of $\Phi_{l}$ contains $W_{i,i}^{(1)}(0)$ for any $i$ since we have
\begin{equation}
	\Phi_{l} \left( \sum_{i=0}^{n-1} H_{i,1} + \dfrac{\hbar}{2} \sum_{i=1}^{n-1} i H_{i,0} - \dfrac{\hbar}{2} \sum_{i=0}^{n-1} H_{i,0}^{2} \right) = (-\hbar) \left( -\alpha W_{n,n}^{(1)}(0) - \dfrac{(l\alpha)^2}{2} \right). \label{eq:Heisenberg1}
\end{equation}
We show that the image of $\Phi_{l}$ contains $W_{i,i}^{(1)}(m)$ for any $i$ and $m \neq 0$.
The formula (\ref{eq:image_of_H}) for $i = 0$ shows that the image of $\Phi_{l}$ contains
\begin{equation*}
	\begin{split}
		H' &= W_{n,n}^{(2)}(0) - W_{1,1}^{(2)}(0) + W_{n,n}^{(1)}(0) \Big( W_{1,1}^{(1)}(0) - \alpha \Big) \\
		&\quad - \sum_{m' \geq 0} \left( W_{n,n}^{(1)}(-m') W_{n,n}^{(1)}(m') - W_{1,1}^{(1)}(-m'-1) W_{1,1}^{(1)}(m'+1) \right).
	\end{split}
\end{equation*}
Hence the assertion follows from
\begin{equation}
	[H', W_{1,1}^{(1)}(m)-W_{2,2}^{(1)}(m)] = -m\alpha W_{1,1}^{(1)}(m). \label{eq:Heisenberg2}
\end{equation}
If $l=1$, this completes the proof. 
Suppose $l \geq 2$.
By Proposition~\ref{prop:image} together with the fact that the image of $\Phi_{l}$ contains $W_{i,j}^{(1)}(m)$ for any $i,j,m$, the image of $\Phi_{l}$ contains
\[
	\left\{ W_{n,1}^{(r)}(1),\ W_{1,n}^{(r)}(-1),\ W_{i,i+1}^{(r)}(0),\ W_{i+1,i}^{(r)}(0) \mid r=1,2 \text{ and } i=1,\ldots,n-1 \right\}.
\]
The proof is complete by Proposition~\ref{prop:generators2}.

\begin{rem}
A proof of the surjectivity of $\Phi_{1} = \ev$ was initially given by the first named author in \cite{MR3923494} by a different method.
The above argument, just computing (\ref{eq:Heisenberg2}), supplies a much simpler proof (the computation of (\ref{eq:Heisenberg1}) for $l=1$ was already appeared in \cite{MR3923494}). 
\end{rem}

\section{}\label{section:commutativity}

\begin{lem}\label{lem:commutativity1}
We have
\begin{equation*}
	\begin{split}
		[W_{i,i}^{(2)}(0),W_{j,j}^{(2)}(0)] &= - W_{i,i}^{(2)}(0) + W_{j,j}^{(2)}(0)\\
		&\qquad + \sum_{m \geq 0} \Big( W_{i,j}^{(2)} (-m) W_{j,i}^{(1)} (m) + W_{j,i}^{(1)} (-m-1) W_{i,j}^{(2)} (m+1) \Big) \\
		&\qquad \ - \sum_{m \geq 0} \Big( W_{j,i}^{(2)} (-m) W_{i,j}^{(1)} (m) + W_{i,j}^{(1)} (-m-1) W_{j,i}^{(2)} (m+1) \Big)\\
		&\qquad \ \ + (l-1)\alpha \sum_{m \geq 1} m \Big( W_{j,i}^{(1)} (-m) W_{i,j}^{(1)} (m) - W_{i,j}^{(1)} (-m) W_{j,i}^{(1)} (m) \Big)\\
		&\qquad \ \ \ + (l-1) \sum_{m \geq 1} m \Big( W_{i,i}^{(1)} (-m) W_{j,j}^{(1)} (m) - W_{j,j}^{(1)} (-m) W_{i,i}^{(1)} (m) \Big).
	\end{split}
\end{equation*}
\end{lem}

\begin{lem}\label{lem:commutativity2}
For $i < j$, we have
\begin{equation*}
	\begin{split}
		&[W_{i,i}^{(2)}(0), A_{j}+B_{j}] = \sum_{m \geq 0} \Big( -W_{j,i}^{(2)}(-m) W_{i,j}^{(1)}(m) + W_{j,i}^{(1)}(-m) W_{i,j}^{(2)}(m) \Big)\\
		&\ + (l-1)\alpha \sum_{m \geq 1} m W_{j,i}^{(1)}(-m) W_{i,j}^{(1)}(m) + (l-1) \sum_{m \geq 1} m \Big( W_{i,i}^{(1)}(-m) W_{j,j}^{(1)}(m) - W_{j,j}^{(1)}(-m) W_{i,i}^{(1)}(m) \Big)
	\end{split}
\end{equation*}
and
\begin{equation*}
	\begin{split}
		&[W_{j,j}^{(2)}(0), A_{i}+B_{i}] = \sum_{m \geq 0} \Big( -W_{i,j}^{(2)}(-m-1) W_{j,i}^{(1)}(m+1) + W_{i,j}^{(1)}(-m-1) W_{j,i}^{(2)}(m+1) \Big)\\
		&\ + (l-1)\alpha \sum_{m \geq 1} m W_{i,j}^{(1)}(-m) W_{j,i}^{(1)}(m) + (l-1) \sum_{m \geq 1} m \Big( W_{j,j}^{(1)}(-m) W_{i,i}^{(1)}(m) - W_{i,i}^{(1)}(-m) W_{j,j}^{(1)}(m) \Big).
	\end{split}
\end{equation*}
\end{lem}

\begin{proof}[Proof of Proposition~\ref{prop:commuting_elements}]

We may assume $i < j$.
Then, by Lemma~\ref{lem:commutativity1}, \ref{lem:commutativity2}, and (\ref{eq:AB}), we have
\begin{equation*}
	\begin{split}
		[D_{i}, D_{j}] &= [W_{i,i}^{(2)}(0), W_{j,j}^{(2)}(0)] - [W_{i,i}^{(2)}(0), A_{j}+B_{j}] + [W_{j,j}^{(2)}(0), A_{i}+B_{i}] + [A_{i}+B_{i},A_{j}+B_{j}] \\
		&= - W_{i,i}^{(2)}(0) + W_{j,j}^{(2)}(0) + W_{i,j}^{(2)}(0) W_{j,i}^{(1)}(0) - W_{j,i}^{(1)}(0) W_{i,j}^{(2)}(0) \\
		&= - W_{i,i}^{(2)}(0) + W_{j,j}^{(2)}(0) + [W_{i,j}^{(2)}(0), W_{j,i}^{(1)}(0)].
	\end{split}
\end{equation*}
This is equal to $0$ by Lemma~\ref{lem:OPE2-1}.
\end{proof}

\section{}\label{section:proof_of_evaluation_map}

The equality (\ref{eq:AB}) is deduced from the following.
\begin{lem}
For $i < j$, we have
\begin{equation*}
	\begin{split}
		&[A_{i}, A_{j}] = \sum_{\substack{m,m' \geq 0\\ m-m'>0}} \sum_{a=1}^{i-1} \Big( W_{j,a}^{(1)}(-m') W_{i,j}^{(1)}(-m+m') W_{a,i}^{(1)}(m) - W_{i,a}^{(1)}(-m) W_{j,i}^{(1)}(m-m') W_{a,j}^{(1)}(m') \Big),
	\end{split}
\end{equation*}
\[
	[A_{i}, B_{j}]=0,
\]
\begin{equation*}
	\begin{split}
		&[B_{i}, A_{j}] = \sum_{m,m' \geq 0} \\
		&\Bigg( \sum_{a=1}^{i-1} \Big( -W_{j,a}^{(1)}(-m') W_{i,j}^{(1)}(-m-1) W_{a,i}^{(1)}(m+m'+1) + W_{i,a}^{(1)}(-m-m'-1) W_{j,i}^{(1)}(m+1) W_{a,j}^{(1)}(m') \Big)\\
		&\ + \sum_{a=j}^{n} \Big( -W_{j,a}^{(1)}(-m-m'-1) W_{i,j}^{(1)}(m') W_{a,i}^{(1)}(m+1) + W_{i,a}^{(1)}(-m-1) W_{j,i}^{(1)}(-m') W_{a,j}^{(1)}(m+m'+1) \Big) \Bigg)\\
		&\quad + \sum_{m \geq 1} m \Big( -W_{i,i}^{(1)}(-m) W_{j,j}^{(1)}(m) + W_{j,j}^{(1)}(-m) W_{i,i}^{(1)}(-m) \Big), 
	\end{split}
\end{equation*}
\begin{equation*}
	\begin{split}
		[B_{i}, B_{j}] &= \sum_{\substack{m,m' \geq 0\\ -m+m' \geq 0}} \sum_{a=j}^{n} \Big( W_{j,a}^{(1)}(-m'-1) W_{i,j}^{(1)}(-m+m') W_{a,i}^{(1)}(m+1) \\
		&\qquad\qquad\qquad\qquad - W_{i,a}^{(1)}(-m-1) W_{j,i}^{(1)}(m-m') W_{a,j}^{(1)}(m'+1) \Big)\\
		&\quad + l \sum_{m \geq 1} m \Big( W_{i,i}^{(1)}(-m) W_{j,j}^{(1)}(m) - W_{j,j}^{(1)}(-m) W_{i,i}^{(1)}(-m) \Big).
	\end{split}
\end{equation*}
\end{lem}
\begin{proof}
In the proof, we write $W_{i,j}^{(1)}(m)$ as $W_{i,j}(m)$.
We have
\begin{equation*}
	\begin{split}
		&[A_{i},A_{j}] = \sum_{m,m' \geq 0} \sum_{a=1}^{i-1} \sum_{b=1}^{j-1} [ W_{i,a}(-m) W_{a,i}(m), W_{j,b}(-m') W_{b,j}(m') ]\\
		&= \sum_{m,m' \geq 0} \sum_{a=1}^{i-1} \sum_{b=1}^{j-1} \\
		&\quad \Big( [W_{i,a}(-m), W_{j,b}(-m')] W_{a,i}(m) W_{b,j}(m') + W_{i,a}(-m) [W_{a,i}(m), W_{j,b}(-m')] W_{b,j}(m') \\
		&\ \quad + W_{j,b}(-m') [W_{i,a}(-m), W_{b,j}(m')] W_{a,i}(m) + W_{j,b}(-m') W_{i,a}(-m) [ W_{a,i}(m), W_{b,j}(m') ] \Big)\\
		&= \sum_{m,m' \geq 0} \sum_{a=1}^{i-1} \Big( -W_{j,a}(-m-m') W_{a,i}(m) W_{i,j}(m') - W_{i,a}(-m) W_{j,i}(m-m') W_{a,j}(m') \\
		&\qquad\qquad\qquad + W_{j,a}(-m') W_{i,j}(-m+m') W_{a,i}(m) + W_{j,i}(-m') W_{i,a}(-m) W_{a,j}(m+m') \Big).
	\end{split} 
\end{equation*}
Then the first assertion follows from
\begin{equation*}
	\begin{split}
		&\sum_{m,m' \geq 0} \sum_{a=1}^{i-1} \Big( -W_{j,a}(-m-m') W_{i,j}(m') W_{a,i}(m) + W_{j,a}(-m') W_{i,j}(-m+m') W_{a,i}(m) \\
		&\qquad\qquad\qquad -W_{j,a}(-m-m') [W_{a,i}(m), W_{i,j}(m')] \Big) \\
		&= \sum_{\substack{m,m' \geq 0\\ m-m' > 0}} \sum_{a=1}^{i-1} W_{j,a}(-m') W_{i,j}(-m+m') W_{a,i}(m) - \sum_{m,m' \geq 0} \sum_{a=1}^{i-1} W_{j,a}(-m-m') W_{a,j}(m+m')
	\end{split} 
\end{equation*}
and
\begin{equation*}
	\begin{split}
		&\sum_{m,m' \geq 0} \sum_{a=1}^{i-1} \Big( - W_{i,a}(-m) W_{j,i}(m-m') W_{a,j}(m') + W_{i,a}(-m) W_{j,i}(-m') W_{a,j}(m+m') \\
		&\qquad\qquad\qquad + [W_{j,i}(-m'), W_{i,a}(-m)] W_{a,j}(m+m') \Big) \\
		&= -\sum_{\substack{m,m' \geq 0\\ m-m' > 0}} \sum_{a=1}^{i-1} W_{i,a}(-m) W_{j,i}(m-m') W_{a,j}(m') + \sum_{m,m' \geq 0} \sum_{a=1}^{i-1} W_{j,a}(-m-m') W_{a,j}(m+m').
	\end{split} 
\end{equation*}

We have $[A_{i},B_{j}] = 0$ due to $i < j$.

We have
\begin{equation*}
	\begin{split}
		&[B_{i},A_{j}] = \sum_{m,m' \geq 0} \sum_{a=i}^{n} \sum_{b=1}^{j-1} [ W_{i,a}(-m-1) W_{a,i}(m+1), W_{j,b}(-m') W_{b,j}(m') ]\\
		&= \sum_{m,m' \geq 0} \\
		&\Bigg( \sum_{b=1}^{j-1} W_{i,b}(-m-m'-1) W_{j,i}(m+1) W_{b,j}(m') - \sum_{a=i}^{n} W_{j,a}(-m-m'-1) W_{a,i}(m+1) W_{i,j}(m') \\
		&\ + \sum_{a=i}^{j-1} \Big( - W_{i,a}(-m-1) W_{j,i}(m-m'+1) W_{a,j}(m') + W_{j,a}(-m') W_{i,j}(-m+m'-1) W_{a,i}(m+1) \Big) \\
		&\ \ + \sum_{a=i}^{n} W_{j,i}(-m') W_{i,a}(-m-1) W_{a,j}(m+m'+1) - \sum_{b=1}^{j-1} W_{j,b}(-m') W_{i,j}(-m-1) W_{b,i}(m+m'+1) \Bigg).
	\end{split} 
\end{equation*}
Then the third assertion follows from
\begin{equation*}
	\begin{split}
		&\sum_{m,m' \geq 0} \\
		&\Bigg( \sum_{b=1}^{j-1} W_{i,b}(-m-m'-1) W_{j,i}(m+1) W_{b,j}(m') - \sum_{a=i}^{j-1} W_{i,a}(-m-1) W_{j,i}(m-m'+1) W_{a,j}(m')\\
		&\ + \sum_{a=i}^{n} \Big( W_{i,a}(-m-1) W_{j,i}(-m') W_{a,j}(m+m'+1) + [W_{j,i}(-m'), W_{i,a}(-m-1)] W_{a,j}(m+m'+1) \Big) \Bigg)\\
		&= \sum_{m,m' \geq 0} \\
		&\Bigg( \sum_{a=i}^{j-1} \Big( W_{i,a}(-m-m'-1) W_{j,i}(m+1) W_{a,j}(m') - W_{i,a}(-m-1) W_{j,i}(m-m'+1) W_{a,j}(m')\\
		&\qquad\qquad + W_{i,a}(-m-1) W_{j,i}(-m') W_{a,j}(m+m'+1) \Big) \\
		&\ + \sum_{a=1}^{i-1} W_{i,a}(-m-m'-1) W_{j,i}(m+1) W_{a,j}(m') + \sum_{a=j}^{n} W_{i,a}(-m-1) W_{j,i}(-m') W_{a,j}(m+m'+1)\\
		&\ \qquad + \sum_{a=i}^{n} [W_{j,i}(-m'), W_{i,a}(-m-1)] W_{a,j}(m+m'+1) \Bigg)\\
		&= \sum_{m,m' \geq 0} \\
		&\Bigg( \sum_{a=1}^{i-1} W_{i,a}(-m-m'-1) W_{j,i}(m+1) W_{a,j}(m') + \sum_{a=j}^{n} W_{i,a}(-m-1) W_{j,i}(-m') W_{a,j}(m+m'+1)\\
		&\qquad + \sum_{a=i}^{n} W_{j,a}(-m-m'-1) W_{a,j}(m+m'+1) - W_{i,i}(-m-m'-1) W_{j,j}(m+m'+1) \Bigg)
	\end{split} 
\end{equation*}
and
\begin{equation*}
	\begin{split}
		&\sum_{m,m' \geq 0} \\
		&\Bigg( -\sum_{a=i}^{n} \Big( W_{j,a}(-m-m'-1) W_{i,j}(m') W_{a,i}(m+1) + W_{j,a}(-m-m'-1) [W_{a,i}(m+1), W_{i,j}(m')] \Big)\\
		&\ + \sum_{a=i}^{j-1} W_{j,a}(-m') W_{i,j}(-m+m'-1) W_{a,i}(m+1) - \sum_{b=1}^{j-1} W_{j,b}(-m') W_{i,j}(-m-1) W_{b,i}(m+m'+1) \Bigg)\\
		&= \sum_{m,m' \geq 0} \\
		&\Bigg( \sum_{a=i}^{j-1} \Big( -W_{j,a}(-m-m'-1) W_{i,j}(m') W_{a,i}(m+1) + W_{j,a}(-m') W_{i,j}(-m+m'-1) W_{a,i}(m+1) \\
		&\qquad\qquad - W_{j,a}(-m') W_{i,j}(-m-1) W_{a,i}(m+m'+1) \Big)\\
		&\ - \sum_{a=j}^{n} W_{j,a}(-m-m'-1) W_{i,j}(m') W_{a,i}(m+1) - \sum_{a=1}^{i-1} W_{j,a}(-m') W_{i,j}(-m-1) W_{a,i}(m+m'+1) \\
		&\ \qquad - \sum_{a=i}^{n} W_{j,a}(-m-m'-1) [W_{a,i}(m+1), W_{i,j}(m')] \Bigg) \\
		&= \sum_{m,m' \geq 0} \\
		&\Bigg( - \sum_{a=j}^{n} W_{j,a}(-m-m'-1) W_{i,j}(m') W_{a,i}(m+1) - \sum_{a=1}^{i-1} W_{j,a}(-m') W_{i,j}(-m-1) W_{a,i}(m+m'+1) \\
		&\qquad - \sum_{a=i}^{n} W_{j,a}(-m-m'-1) W_{a,j}(m+m'+1) + W_{j,j}(-m-m'-1) W_{i,i}(m+m'+1) \Bigg).
	\end{split} 
\end{equation*}

The last assertion is computed as
\begin{equation*}
	\begin{split}
		&[B_{i},B_{j}] = 
		\sum_{m,m' \geq 0} \sum_{b=j}^{n} \\
		&\quad \Big( W_{i,b}(-m-m'-2) W_{j,i}(m+1) W_{b,j}(m'+1) - W_{i,b}(-m-1) W_{j,i}(m-m') W_{b,j}(m'+1) \\
		&\ + W_{j,b}(-m'-1) W_{i,j}(-m+m') W_{b,i}(m+1) - W_{j,b}(-m'-1) W_{i,j}(-m-1) W_{b,i}(m+m'+2) \Big)\\
		&\qquad + l \sum_{m \geq 0} (m+1) \Big( W_{i,i}(-m-1) W_{j,j}(m+1) - W_{j,j}(-m-1) W_{i,i}(m+1) \Big)\\
		&= \sum_{\substack{m,m' \geq 0\\ m-m' \leq 0}} \sum_{b=j}^{n} \\
		&\quad \Big( - W_{i,b}(-m-1) W_{j,i}(m-m') W_{b,j}(m'+1) + W_{j,b}(-m'-1) W_{i,j}(-m+m') W_{b,i}(m+1) \Big)\\
		&\qquad + l \sum_{m \geq 1} m \Big( W_{i,i}(-m) W_{j,j}(m) - W_{j,j}(-m) W_{i,i}(m) \Big).
	\end{split} 
\end{equation*}
\end{proof}

\def\cprime{$'$} \def\cprime{$'$} \def\cprime{$'$} \def\cprime{$'$}
\providecommand{\bysame}{\leavevmode\hbox to3em{\hrulefill}\thinspace}
\providecommand{\MR}{\relax\ifhmode\unskip\space\fi MR }
\providecommand{\MRhref}[2]{%
  \href{http://www.ams.org/mathscinet-getitem?mr=#1}{#2}
}
\providecommand{\href}[2]{#2}


\begin{thebibliography}{FFNR}

\bibitem[AGT]{MR2586871}
Luis~Fernando Alday, Davide Gaiotto, and Yuji Tachikawa, \emph{Liouville
  correlation functions from four-dimensional gauge theories}, Lett. Math.
  Phys. \textbf{91} (2010), no.~2, 167--197. 

\bibitem[A1]{MR2318558}
Tomoyuki Arakawa, \emph{Representation theory of {$\mathscr{W}$}-algebras},
  Invent. Math. \textbf{169} (2007), no.~2, 219--320. 

\bibitem[A2]{MR3966808}
\bysame, \emph{Representation theory of {$W$}-algebras and {H}iggs branch
  conjecture}, Proceedings of the {I}nternational {C}ongress of
  {M}athematicians---{R}io de {J}aneiro 2018. {V}ol. {II}. {I}nvited lectures,
  World Sci. Publ., Hackensack, NJ, 2018, pp.~1263--1281. 

\bibitem[AM]{MR3598875}
Tomoyuki Arakawa and Alexander Molev, \emph{Explicit generators in rectangular
  affine {$\mathcal{W}$}-algebras of type {$A$}}, Lett. Math. Phys.
  \textbf{107} (2017), no.~1, 47--59. 

\bibitem[BFFR]{MR2851149}
Alexander Braverman, Boris Feigin, Michael Finkelberg, and Leonid Rybnikov,
  \emph{A finite analog of the {AGT} relation {I}: {F}inite {$W$}-algebras and
  quasimaps' spaces}, Comm. Math. Phys. \textbf{308} (2011), no.~2, 457--478.

\bibitem[BFN]{MR3592485}
Alexander Braverman, Michael Finkelberg, and Hiraku Nakajima, \emph{Instanton
  moduli spaces and $\mathscr{W}$-algebras}, Ast\'{e}risque (2016), no.~385,
  vii+128. 

\bibitem[BK1]{MR2199632}
Jonathan Brundan and Alexander Kleshchev, \emph{Shifted {Y}angians and finite
  {$W$}-algebras}, Adv. Math. \textbf{200} (2006), no.~1, 136--195.

\bibitem[BK2]{MR2456464}
\bysame, \emph{Representations of shifted {Y}angians and finite
  {$W$}-algebras}, Mem. Amer. Math. Soc. \textbf{196} (2008), no.~918,
  viii+107. 

\bibitem[CH1]{MR3925243}
Thomas Creutzig and Yasuaki Hikida, \emph{Rectangular {W}-algebras, extended higher spin gravity and dual
  coset {CFT}s}, J. High Energy Phys. (2019), no.~2, 147, front matter + 30.

\bibitem[CH2]{MR4032024}
\bysame, \emph{Rectangular {$W$} algebras and
  superalgebras and their representations}, Phys. Rev. D \textbf{100} (2019),
  no.~8, 086008, 27. 

\bibitem[EP]{MR4061286}
Lorenz Eberhardt and Tom\'{a}\v{s} Proch\'{a}zka, \emph{The matrix-extended
  {$\mathcal{W}_{1+\infty}$} algebra}, J. High Energy Phys. (2019), no.~12,
  175, 34 pages. 

\bibitem[FFNR]{MR2827177}
Boris Feigin, Michael Finkelberg, Andrei Negut, and Leonid Rybnikov,
  \emph{Yangians and cohomology rings of {L}aumon spaces}, Selecta Math. (N.S.)
  \textbf{17} (2011), no.~3, 573--607.

\bibitem[FF]{MR1071340}
Boris Feigin and Edward Frenkel, \emph{Quantization of the {D}rinfel\cprime
  d-{S}okolov reduction}, Phys. Lett. B \textbf{246} (1990), no.~1-2, 75--81.

\bibitem[FT]{MR3971731}
Michael Finkelberg and Alexander Tsymbaliuk, \emph{Multiplicative slices,
  relativistic {T}oda and shifted quantum affine algebras}, Representations and
  nilpotent orbits of {L}ie algebraic systems, Progr. Math., vol. 330,
  Birkh\"{a}user/Springer, Cham, 2019, pp.~133--304. 

\bibitem[FB]{MR2082709}
Edward Frenkel and David Ben-Zvi, \emph{Vertex algebras and algebraic curves},
  second ed., Mathematical Surveys and Monographs, vol.~88, American
  Mathematical Society, Providence, RI, 2004. 

\bibitem[Ge1]{MR3663604}
Naoki Genra, \emph{Screening operators for {$\mathcal{W}$}-algebras}, Selecta
  Math. (N.S.) \textbf{23} (2017), no.~3, 2157--2202. 

\bibitem[Ge2]{MR4091897}
Naoki Genra, \emph{Screening operators and parabolic inductions for affine
  $\mathcal{W}$-algebras {\rm (}with an appendix by {S}higenori {N}akatsuka{\rm )}}, Adv. Math. \textbf{369} (2020), 107179, 62 pages. 

\bibitem[Gu]{MR2323534}
Nicolas Guay, \emph{Affine {Y}angians and deformed double current algebras in type {A}}, Adv. Math. \textbf{211} (2007), no.~2, 436--484.

\bibitem[GNW]{MR3861718}
Nicolas Guay, Hiraku Nakajima, and Curtis Wendlandt, \emph{Coproduct for
  {Y}angians of affine {K}ac-{M}oody algebras}, Adv. Math. \textbf{338} (2018),
  865--911. 

\bibitem[GRW]{MR4014633}
Nicolas Guay, Vidas Regelskis, and Curtis Wendlandt, \emph{Vertex
  representations for {Y}angians of {K}ac-{M}oody algebras}, J. \'{E}c.
  polytech. Math. \textbf{6} (2019), 665--706. 

\bibitem[KRW]{MR2013802}
Victor Kac, Shi-Shyr Roan, and Minoru Wakimoto, \emph{Quantum reduction for
  affine superalgebras}, Comm. Math. Phys. \textbf{241} (2003), no.~2-3,
  307--342. 

\bibitem[K1]{MR3923494}
Ryosuke Kodera, \emph{Braid group action on affine {Y}angian}, SIGMA Symmetry
  Integrability Geom. Methods Appl. \textbf{15} (2019), 020, 28 pages.

\bibitem[K2]{MR4207398}
\bysame, \emph{On {G}uay's evaluation map for affine {Y}angians}, Algebr.
  Represent. Theory \textbf{24} (2021), no.~1, 253--267, correction 269--272, arXiv:1806.09884. 

\bibitem[MO]{MR3951025}
Davesh Maulik and Andrei Okounkov, \emph{Quantum groups and quantum
  cohomology}, Ast\'{e}risque (2019), no.~408, ix+209. 

\bibitem[Nakaj]{MR3024827}
Hiraku Nakajima, \emph{Handsaw quiver varieties and finite {$W$}-algebras},
  Mosc. Math. J. \textbf{12} (2012), no.~3, 633--666, 669--670. 

\bibitem[Nakat]{2005.10472}
Shigenori Nakatsuka, \emph{On {M}iura maps for {$\mathcal{W}$}-superalgebras},
   arXiv:2005.10472.

\bibitem[Ne1]{Toward_AGT_for_parabolic_sheaves}
Andrei Negut, \emph{Toward {AGT} for parabolic sheaves}, arXiv:1911.02963,
 to appear in IMRN, https://doi.org/10.1093/imrn/rnaa308.

\bibitem[Ne2]{Deformed_W_algebras_in_type_A_for_rectangular_nilpotent}
\bysame, \emph{Deformed {$W$}-algebras in type {A} for rectangular
  nilpotent}, arXiv:2004.02737.

\bibitem[RS]{MR1700166}
Eric Ragoucy and Paul Sorba, \emph{Yangian realisations from finite
  $\mathcal{W}$-algebras}, Comm. Math. Phys. \textbf{203} (1999), no.~3,
  551--572. 

\bibitem[R]{MR4088277}
Miroslav Rap\v{c}\'{a}k, \emph{On extensions of
  {$\widehat{\mathfrak{gl}(m|n)}$} {K}ac-{M}oody algebras and {C}alabi-{Y}au
  singularities}, J. High Energy Phys. (2020), no.~1, 042, 34 pages.

\bibitem[SV]{MR3150250}
Olivier Schiffmann and Eric Vasserot, \emph{Cherednik algebras, {W}-algebras
  and the equivariant cohomology of the moduli space of instantons on
  {$\bold{A}^2$}}, Publ. Math. Inst. Hautes \'Etudes Sci. \textbf{118} (2013),
  213--342.

\bibitem[U1]{Affine_super_Yangian}
Mamoru Ueda, \emph{Construction of affine super {Y}angian}, arXiv:1911.06666, to appear in Publ.\ RIMS.

\bibitem[U2]{Affine_super_Yangians_and_rectangular_W-superalgebras}
\bysame, \emph{Affine super {Y}angians and rectangular
  {$W$}-superalgebras}, arXiv:2002.03479.

\bibitem[VV]{2011.01203}
Michela Varagnolo and Eric Vasserot, \emph{K-theoretic {H}all algebras, quantum
  groups and super quantum groups}, arXiv:2011.01203.
\end{thebibliography}
\end{document}